\documentclass[11pt,a4paper]{article}

\usepackage[a4paper,margin = 2.5cm]{geometry}
\usepackage{amsmath,amsthm,amssymb}
\usepackage{enumerate,paralist}
\usepackage{multicol}
\usepackage{graphicx}
\usepackage{pictex}
\usepackage{tikz}
\usepackage{ifpdf}   
\AtBeginDocument{\gdef~{\nobreakspace{}}\catcode`\"=12}  
\usepackage[all,tips]{xy}  
\CompileMatrices  
\usepackage[utf8]{inputenc}
\usepackage{fontenc}             

\newcommand{\thmref}[1]{Theorem~\ref{#1}}
\newcommand{\secref}[1]{\S\ref{#1}}
\newcommand{\lemref}[1]{Lemma~\ref{#1}}

\newcommand{\figref}[1]{Figure~\ref{#1}}

\newcommand{\remarkref}[1]{Remark~\ref{#1}}
\newcommand{\itemref}[1]{\ref{#1}}   

\newcommand{\implica}{\mathbin{\rightarrow}}

\newcommand{\bl}{\mathbf}

\newcommand {\MG} {\mathbb{MG}}
\newcommand {\G} {\mathbb{G}}

\DeclareMathOperator{\X}{X}
\DeclareMathOperator{\D}{I}

\theoremstyle{plain}

\newtheorem{thm}{Theorem}[section]
\newtheorem{prop}[thm]{Proposition}
\newtheorem{lem}[thm]{Lemma}
\newtheorem{cor}[thm]{Corollary}

\theoremstyle{definition}
\newtheorem{dfn}[thm]{Definition}
\newtheorem{example}[thm]{Example}

\theoremstyle{remark}
\newtheorem{remark}[thm]{Remark}

\begin{document}

\bibliographystyle{plain} 
\title{An algebraic study of S5-modal Gödel logic}
\author{Diego Castaño, Cecilia Cimadamore, José Patricio Díaz Varela, Laura Rueda}
\maketitle

\begin{abstract}
In this paper we continue the study of the variety $\mathbb{MG}$ of monadic Gödel algebras. These algebras are the equivalent algebraic semantics of the S5-modal expansion of Gödel logic, which is equivalent to the one-variable monadic fragment of first-order Gödel logic. We show three families of locally finite subvarieties of $\mathbb{MG}$ and give their equational bases. We also introduce a topological duality for monadic Gödel algebras and, as an application of this representation theorem, we characterize congruences and give characterizations of the locally finite subvarieties mentioned above by means of their dual spaces. Finally, we study some further properties of the subvariety generated by monadic Gödel chains: we present a characteristic chain for this variety, we prove that a Glivenko-type theorem holds for these algebras and we characterize free algebras over $n$ generators. 
\end{abstract}

\section{Introduction and preliminaries}

\newcounter{saveenum_mbl}

In this article we assume a general knowledge of Hájek's Basic Logic and its equivalent algebraic semantics, the variety of BL-algebras (see \cite{Hajek98}).

In \cite{Hajek98} Hájek introduced the S5-modal expansion $S5(\mathcal{C})$ of any axiomatic extension $\mathcal{C}$ of his Basic Logic $\mathcal{BL}$. This logic is defined on the language of basic logic augmented with the unary connectives $\square$ and $\lozenge$ by interpreting formulas on structures based on $\mathcal{C}$-chains. This expansion is particularly interesting because it is equivalent to the one-variable fragment of the first-order extension of $\mathcal{C}$ (see \cite{Hajek98}).

One of the most important extensions of $\mathcal{BL}$ is Gödel logic $\mathcal{G}$, obtained by adding the axiom $p \to p^2$. This logic has been studied in recent papers (e.g. \cite{CR15,CCDVR17,CCDVR20}).  In \cite{CCDVR20} we showed a strong completeness theorem for $S5(\mathcal{G})$ based on an algebraic semantics for $S5(\mathcal{G})$. This algebraic semantics is the variety $\mathbb{MG}$ of {\em monadic Gödel algebras}. We started the study of this variety in $\cite{CCDVR20}$; the main objective of this article is to expand our understanding of this variety, especially of some of its subvarieties, which naturally correspond to axiomatic extensions of $S5(\mathcal{G})$.

$\mathbb{MG}$ is a subvariety of the variety $\mathbb{MBL}$ of {\em monadic BL-algebras} introduced in \cite{CCDVR17}. Monadic BL-algebras are BL-algebras endowed with two unary operations $\forall$ and $\exists$ that satisfy the following identities:
\begin{multicols}{2}
\begin{enumerate}[(M1)]
\item\label{M1} $\forall x\implica x\approx 1$.
\item\label{M2} $\forall ( x\implica \forall y)\approx \exists x\implica \forall y$.
\item \label{M3}$\forall (\forall x\implica y)\approx \forall x\implica \forall y$.
\item\label{M4} $\forall ( \exists x\vee y)\approx \exists x\vee \forall y$.
\item\label{M5} $\exists (x*x)\approx \exists x*\exists x$.

\phantom{aaaa}
\setcounter{saveenum_mbl}{\value{enumi}}
\end{enumerate}
\end{multicols}
The class $\mathbb{MBL}$ was defined as a candidate for the equivalent algebraic semantics of $S5(\mathcal{BL})$; note that we use $\forall$ and $\exists$ instead of $\square$ and $\lozenge$, respectively. The variety $\mathbb{MG}$ is obtained by adding the identity $x^2 \approx x$ to those for $\mathbb{MBL}$ (see \cite{CCDVR17,CCDVR20}). Note also that the identity (M\ref{M5}) becomes trivial when $x^2 \approx x$ holds.


Monadic Heyting algebras were introduced by Monteiro and Varsavksy in \cite{MoVa57} and later studied in depth by Bezhanishvili in \cite{bezhanishvili98}. In \cite{CCDVR17} we also showed that monadic Gödel algebras coincide with monadic Heyting algebras that satisfy the prelinearity identity $(x \implica y) \vee (y \implica x) \approx 1$ and  identity (M4). 
%

The variety of monadic Gödel algebras contains many subvarieties that are interesting to study. We introduce some of these subvarieties in Section 2; specifically we introduce subvarieties based on the notions of height and width of an algebra. The variety $\mathbb{MG}$ is not locally finite; however, all the subvarieties introduced are proved to be locally finite varieties.

The aim of Section 3 is to give a topological representation of monadic Gödel algebras using Priestley spaces. The duality established is based on the duality given by Cignoli in \cite{Cignoli91} for distributive lattices with an additive closure operator. As applications, we characterize congruences on monadic Gödel algebras by means of saturated closed increasing subsets of the dual space, and we describe the dual spaces of the algebras belonging to the subvarieties introduced in Section 2.

We devote Section 4 to study the subvariety generated by monadic Gödel chains in more depth. First we produce a characteristic chain for this subvariety, that is, a totally ordered algebra that generates the whole variety. Then we prove a Glivenko-type theorem for this variety. Recall that Glivenko showed in \cite{glivenko1929} that a propositional formula is provable in the classical propositional logic if and only if its double negation is provable in the intuitionistic propositional logic. This result has an algebraic formulation, that is, the double negation is a homomorphism from each Heyting algebra onto the Boolean algebra of its regular elements. We prove a Glivenko-type theorem (in an algebraic version) establishing a relation between algebras in this variety and the class of monadic Gödel algebras such that the image of $\exists$ is a Boolean algebra. Finally, we close Section 4 with a full description of free algebras over a finite number of generators in this variety; we give a procedure to calculate the dual spaces of the free algebras.

\

We finish this section summarizing the basic properties of monadic Gödel algebras; all the proofs can be found in \cite{CCDVR17} in the broader context of monadic BL-algebras. For brevity, if $\mathbf{A}$ is a Gödel algebra and we enrich it with a monadic structure, we denote the resulting algebra by $\langle \mathbf{A}, \exists, \forall\rangle$. The next lemma collects some of the basic properties that hold true in any monadic Gödel algebra. We abbreviate ``finitely subdirectly irreducible'' as f.s.i.

\begin{lem} \label{LEMA: props basicas}
Let $\langle \mathbf{A}, \exists, \forall\rangle$ be a monadic Gödel algebra. Then:
\begin{enumerate}[$(1)$]
\item $\exists A = \forall A$;
\item $\exists \mathbf{A}$ is a subalgebra of $\mathbf{A}$;
\item $\exists a = \min\{c \in \exists A: c \geq a\}$ and $\forall a = \max\{c \in \exists A: c \leq a\}$ for every $a \in A$;
\item the lattices of congruences of $\langle \mathbf{A}, \exists, \forall\rangle$ and $\exists \mathbf{A}$ are isomorphic;
\item $\langle \mathbf{A}, \exists, \forall\rangle$ is f.s.i.\ if and only if $\exists \mathbf{A}$ is totally ordered;
\item $\langle \mathbf{A}, \exists, \forall\rangle$ is subdirectly irreducible if and only if $\exists \mathbf{A}$ is totally ordered and there exists $u \in \exists A \setminus \{1\}$ such that $a \leq u$ for all $a \in \exists A \setminus \{1\}$.
\end{enumerate}
\end{lem}

The next lemma includes several arithmetical properties that are used constantly throughout the paper.

\begin{lem}
Let $\langle \mathbf{A}, \exists, \forall\rangle$ be a monadic Gödel algebra. Then, for any $a,b \in A$ and $c \in \exists A$:
\begin{multicols}{2}
\begin{enumerate}[$(1)$]
\item $\forall 1 = \exists 1 = 1$ and $\forall 0 = \exists 0 = 0$;
\item $\forall c = \exists c = c$;
\item $\forall a \leq a \leq \exists a$;
\item if $a \leq b$, then $\forall a \leq \forall b$ and $\exists a \leq \exists b$;
\item $\forall (a \vee c) = \forall a \vee c$;
\item $\exists (a \vee b) = \exists a \vee \exists b$;
\item $\forall (a \wedge b) = \forall a \wedge \forall b$;
\item $\exists (a \wedge c) = \exists a \wedge c$;
\item $\forall (a \implica c) = \exists a \implica c$;
\item $\exists (a \implica c) \leq \forall a \implica c$;
\item $\forall (c \implica a) = c \implica \forall a$;
\item $\exists (c \implica a) \leq c \implica \exists a$;
\item $\forall \neg a = \neg \exists a$;
\item $\exists \neg a \leq \neg \forall a$.
\end{enumerate}
\end{multicols}
\end{lem}


In \cite{CCDVR17} we give a characterization of those subalgebras of a given BL-algebra that may be the range of the quantifiers $\forall$ and $\exists$. Given a BL-algebra $\mathbf{A}$, we say that a subalgebra $\mathbf{C} \leq \mathbf{A}$ is {\em $m$-relatively complete} if the following conditions hold:
\begin{enumerate}[(s1)]
\item For every $a \in A$, the subset $\{c \in C: c \leq a\}$ has a greatest element and $\{c \in C: c \geq a\}$ has a least element.
\item For every $a \in A$ and $c_1,c_2 \in C$ such that $c_1 \leq c_2 \vee a$, there exists $c_3 \in C$ such that $c_1 \leq c_2 \vee c_3$ and $c_3 \leq a$.
\item For every $a \in A$ and $c_1 \in C$ such that $a * a \leq c_1$, there exists $c_2 \in C$ such that $a \leq c_2$ and $c_2 * c_2 \leq c_1$.
\end{enumerate}
Under certain circumstances these conditions can be simplified. For example, if $\mathbf{A}$ is finite, condition (s1) is trivially satisfied. If $\mathbf{C}$ is totally ordered, condition (s2) may be replaced by the following simpler equivalent form:
\begin{enumerate}
\item[(s2$_\ell$)] If $1 = c \vee a$ for some $c \in C$, $a \in A$, then $c = 1$ or $a = 1$.
\end{enumerate}
If $\mathbf{A}$ is a Gödel algebra, condition (s3) is immediate since $*$ coincides with $\wedge$. 

Given a BL-algebra $\mathbf{A}$ and an $m$-relatively complete subalgebra $\mathbf{C} \leq \mathbf{A}$, if we define on $A$ the operations $$\exists a := \min \{c \in C: c \geq a\}, \qquad \forall a := \max \{c \in C: c \leq a\},$$ then $\langle \mathbf{A}, \exists, \forall\rangle$ is a monadic BL-algebra such that $\forall A = \exists A = C$.

\section{Some locally finite subvarieties}
\label{sec: subvar loc fin}

In this section we introduce some subvarieties of $\mathbb{MG}$. We prove that there are equations that constrain the ``height'' or ``width'' of an algebra, and that the subvarieties determined by these conditions are locally finite.

We start by recalling that the variety $\mathbb{MG}$ is not locally finite. In fact, there are f.s.i.\ algebras that are not locally finite. For example, consider the monadic Gödel algebra $\mathbf{A} := \langle [0,1]_G^\mathbb{N}, \exists, \forall\rangle$, where $[0,1]_G$ is the standrad Gödel algebra and $\exists A$ is the set of constant sequences in $[0,1]_G^\mathbb{N}$. Let $a \in A$ be the sequence defined by $a(n) := 1 - \frac{1}{n}$, $n \in \mathbb{N}$. We claim that the subuniverse of $\mathbf{A}$ generated by $a$ is infinite. Indeed, consider the following sequences defined by recursion: $a_1 = a$; $a_{k+1} := a_k \vee (a_k \implica \forall a_k)$ for $k \in \mathbb{N}$. It is straightforward to check that $a_k(n) = 1$ for $n < k$, and $a_k(n) = 1 - \frac{1}{n}$ for $n \geq k$. Since all these sequences are different, the subalgebra generated by $a$ in $\mathbf{A}$ is infinite.

We now introduce the notions of ``height'' and ``width'' that allow us to give examples of locally finite subvarieties of $\mathbb{MG}$.

It is known from \cite{HechKat} that the variety of Gödel algebras generated by the chain with $n$ elements is characterized in $\G$ by means of the equation $$\displaystyle\bigvee_{i=1}^n(x_i \implica x_{i+1})\approx 1.$$ We denote by $\mathbb{H}_n$ the subvariety of $\MG$ characterized by this identity. The f.s.i.\ algebras in $\mathbb{H}_n$ are precisely the monadic Gödel algebras $\langle \mathbf{A}, \exists, \forall\rangle$ where $\mathbf{A}$ is a subdirect product of Gödel chains with at most $n$ elements; observe that $\exists \mathbf{A}$ is a Gödel chain with at most $n$ elements. We say that the ``height'' of these algebras is at most $n$. It is worth noting that $\mathbb{H}_2$ is the variety of monadic Boolean algebras.

We can define larger subvarieties of $\mathbb{MG}$ if we only require the subalgebra $\exists \mathbf{A}$ to be of finite height. Let $\mathbb{H}^\exists_n$ be the subvariety of $\MG$ axiomatized by the equation $$\displaystyle\bigvee_{i=1}^n(\exists x_i \implica \exists x_{i+1})\approx 1.$$ A f.s.i.\ monadic Gödel algebra $\langle \mathbf{A}, \exists, \forall\rangle$ belongs to $\mathbb{H}^\exists_n$ if and only if $\exists \mathbf{A}$ is a Gödel chain with at most $n$ elements.

Observe that, for any monadic Gödel algebra $\langle \mathbf{A}, \exists, \forall\rangle$, we have that $\langle \mathbf{A}, \exists, \forall\rangle \in \mathbb{H}^\exists_2$ if and only if $\exists \mathbf{A}$ is
a Boolean algebra. Moreover, $\langle \mathbf{A},\exists,\forall\rangle$ is a f.s.i.\ algebra in $\mathbb{H}^\exists_2$ if and only if $\exists A = \{0,1\}$. Now, since the two-element chain is the only simple Gödel algebra, by Lemma \ref{LEMA: props basicas} (4) $\langle \mathbf{A},\exists,\forall\rangle$ is a simple algebra in $\mathbb{MG}$ if and only if $\exists A = \{0,1\}$. Thus $\mathbb{H}_2^\exists$ is the subvariety of $\mathbb{MG}$ generated by its simple members. It is easy to see that $\mathbb{H}^\exists_2$ is a discriminator variety (in fact, the largest one contained in $\mathbb{MG}$) since the term 
$$t(x,y,z) := (\forall((x \implica y)\wedge (y \implica x)) \wedge z)\vee (\neg\forall((x \implica y)\wedge (y \implica x)) \wedge x)$$ gives the ternary discriminator function on each f.s.i.\ (simple) algebra in  $\mathbb{H}^\exists_2$.

The varieties $\mathbb{H}_n$ and $\mathbb{H}^\exists_n$ are characterized by some bounded ``height''. We can also define subvarieties of $\mathbb{MG}$ based on a notion of ``width''. Consider the identity
\begin{equation}
\bigwedge_{1\leq i<j\leq k+1} \forall(x_i \vee x_j) \implica \bigvee_{i=1}^{k+1} \forall x_i \approx 1 \tag{$\alpha_k$}
\end{equation}
If a monadic Gödel algebra $\mathbf{A}$ satisfies equation $(\alpha_k)$ for some $k$, we say that $\mathbf{A}$ has {\em finite width}; the {\em width} of $\mathbf{A}$ is the least of such $k$. We denote by $\mathbb{W}_k$ the subvariety of $\mathbb{MG}$ determined by equation $(\alpha_k)$. The following result justifies this terminology. An {\em orthogonal set} in a (monadic) Gödel algebra $\mathbf{A}$ is a subset $S \subseteq A \setminus \{1\}$ such that $x \vee y = 1$ for every $x,y \in S$, $x \ne y$.

\begin{thm} \label{TEO: equivalencias para ancho k}
Let $\langle \mathbf{A}, \exists, \forall\rangle$ be a f.s.i.\ monadic Gödel algebra. The following conditions are equivalent:
\begin{enumerate}[$(1)$]
\item $\langle \mathbf{A}, \exists, \forall\rangle$ satisfies equation $(\alpha_k)$;
\item any orthogonal set in $\mathbf{A}$ has at most $k$ elements;
\item there are prime filters $P_1, \ldots, P_r$ in $\mathbf{A}$, $r \leq k$, such that $\bigcap_{i=1}^r P_i = \{1\}$ and $P_i \cap \exists A = \{1\}$ for $1 \leq i \leq r$.
\end{enumerate}
\end{thm}

\begin{proof}
We first prove that $(1)$ implies $(2)$. Assume $(1)$ holds and suppose there is an orthogonal set $S = \{s_1,\ldots,s_k,s_{k+1}\}$ with $k+1$ elements. Then $\forall(s_i \vee s_j) = \forall 1 = 1$ for every $i \ne j$, but $\bigvee_{i=1}^{k+1} \forall s_i = \forall s_{i_0} \leq s_{i_0} < 1$ for some $i_0$, since $\exists \mathbf{A}$ is totally ordered. This contradicts the validity of $(\alpha_k)$.

Now assume condition $(2)$ is true. Let $r \leq k$ be the maximal cardinality of orthogonal sets in $\mathbf{A}$ and fix an $r$-element orthogonal set $S = \{s_1,\ldots,s_r\}$. Consider the sets $P_i = \{x \in A: x \vee s_i = 1\}$. We claim that $P_i$ are the desired prime filters. Indeed, it is clear that $P_i$ are filters on $\mathbf{A}$. To prove that $P_i$ is prime, suppose $x \vee y \in P_i$ but $x,y \notin P_i$. Since $(x \vee s_i) \vee (y \vee s_i) = 1$, the set $S' = \{x \vee s_i, y \vee s_i\} \cup S \setminus \{s_i\}$ is an orthogonal set. Hence $S'$ must contain at most $r$ elements. There are three possibilities: if  $x \vee s_i = s_j$, $i \ne j$, then $x \vee s_i = s_j \vee s_i = 1$, so $x \in P_i$, a contradiction; if $y \vee s_i = s_j$, $i \ne j$, then $y \vee s_i = s_j \vee s_i = 1$, so $y \in P_i$, a contradiction; if $x \vee s_i = y \vee s_i$, then $x \vee s_i = x \vee y \vee s_i = 1$, so $x \in P_i$, a contradiction. Thus $P_i$ is prime for $1 \leq i \leq r$. Moreover, if $x \in \bigcap_{i=1}^r P_i$, it must be that $x = 1$, since otherwise $S \cup \{x\}$ would be an orthogonal set with $r+1$ elements. Finally, note that if $c \in P_i \cap \exists A$, then $c \vee s_i = 1$ and, by condition $(s2)$, $c = 1$. This concludes the proof that $(2)$ implies $(3)$.

Finally we prove that $(3)$ implies $(1)$. By the assumption we may consider $\mathbf{A} \leq \mathbf{A}_1 \times \ldots \times \mathbf{A}_r$ with $r \leq k$, where each $\mathbf{A}_i$ is totally ordered. Let $\bar{a}_1,\ldots,\bar{a}_{k+1} \in A$. Let $\bar{b} = \bigwedge_{j < j'} (\bar{a}_j \vee \bar{a}_{j'})$. For each $i \in \{1,\ldots,r\}$, consider the set $S_i := \{\bar{a}_j(i): 1 \leq j \leq k+1\} \subseteq A_i$ and let $j_i \in \{1,\ldots,k+1\}$ be such that $\bar{a}_{j_i}(i) = \min S_i$. Since $\{j_1,\ldots,j_r\}$ is a proper subset of $\{1,\ldots,k+1\}$, there is $j^* \in \{1,\ldots,k+1\}$ such that $j^* \ne j_i$ for $1 \leq i \leq r$. Thus $\bar{b} \leq \bar{a}_{j^*}$ and we have that $$\bigwedge_{j \ne j'} \forall (\bar{a}_j \vee \bar{a}_{j'}) = \forall \bigwedge_{j \ne j'} (\bar{a}_j \vee \bar{a}_{j'}) = \forall \bar{b} \leq \forall \bar{a}_{j^*} \leq \bigvee_{j=1}^{k+1} \forall \bar{a}_j.$$ This proves that $(\alpha_k)$ holds in $\mathbf{A}$.
\end{proof}

\begin{cor} \label{COR: representacion de algebras de ancho k}
If $\langle \mathbf{A},\exists,\forall\rangle$ is a f.s.i.\ monadic Gödel algebra of width $k$, there are totally ordered Gödel algebras $\mathbf{A}_1,\ldots,\mathbf{A}_k$ and an embedding $\varphi\colon \mathbf{A} \to \mathbf{A}_1 \times \ldots \times \mathbf{A}_k$ such that $\pi_i \circ \varphi|_{\exists A}$ is an embedding of $\exists \mathbf{A}$ into $\mathbf{A}_i$ for $1 \leq i \leq k$ (here $\pi_i$ is the projection on the $i$-th component). 
\end{cor}

\begin{proof}
Since $\langle \mathbf{A},\exists,\forall\rangle$ has width $k$, the previous theorem produces prime filters $P_1,\ldots,P_k$ in $\mathbf{A}$ such that $\bigcap_{i=1}^k P_i = \{1\}$ and $P_i \cap \exists A = \{1\}$ for $1 \leq i \leq k$. We thus get an embedding $\varphi\colon \mathbf{A} \to \mathbf{A}/P_1 \times \ldots \times \mathbf{A}/P_k$ with the desired property.
\end{proof}

\begin{remark}
Note that $\mathbb{W}_1$ is the variety generated by monadic Gödel chains. In addition, observe that equation $(\alpha_1)$ is equivalent to 
\begin{enumerate}[(M1)]
\setcounter{enumi}{\value{saveenum_mbl}}
\item\label{axiomac} $\forall(x\vee y) \approx \forall x \vee \forall y$.
\setcounter{saveenum_mbl}{\value{enumi}}
\end{enumerate}
In \cite{CCDVR17} we already proved that the variety generated by chains is axiomatized by this identity in the broader context of monadic BL-algebras.
\end{remark}

\begin{remark}
Observe that, $\mathbf{A} \in \mathbb{H}^\exists_2 \cap \mathbb{W}_1$ is f.s.i.\ if and only if $\mathbf{A}$ is a chain and $\exists A = \{0,1\}$. Then the lattice of subvarieties of $\mathbb{H}^\exists_2 \cap  \mathbb{W}_1$ is isomorphic to that of $\mathbb{G}$, that is, it is a chain of type $\mathbb{N} \oplus \{1\}$. 
\end{remark}

We now prove that all of the subvarieties here introduced are locally finite varieties.

\begin{prop}
The variety $\mathbb{H}_n^\exists$ is locally finite for every $n$.
\end{prop}

\begin{proof}
Since the variety of Gödel algebras is locally finite, for each natural number $m$ the size of an $m$-generated Gödel algebra has an upper bound $N_m$. Now let $\langle \mathbf{A}, \exists, \forall\rangle$ be a f.s.i.\ monadic Gödel algebra in $\mathbb{H}_n^\exists$ generated by a set $S$ of size $m$. We know that $\exists A$ has at most $n$ elements. Thus $\mathbf{A}$ is a Gödel algebra generated by the set $S \cup \exists A$, which has at most $m+n$ elements. Therefore the size of $\mathbf{A}$ is bound by $N_{m+n}$. This shows that the class of f.s.i.\ algebras in $\mathbb{H}_n^\exists$ is uniformly locally finite. Thus, $\mathbb{H}_n^\exists$ is a locally finite variety (see \cite[Theorem 3.7]{bez01}).
\end{proof}

Since $\mathbb{H}_n \subseteq \mathbb{H}_n^\exists$ we also get the following result.

\begin{cor}
The variety $\mathbb{H}_n$ is locally finite for every $n$.
\end{cor}

We now prove that $\mathbb{W}_k$ is also a locally finite variety for every $k$.

\begin{thm} \label{TEO: ancho k son loc fin - caso fsi}
The class of f.s.i.\ algebras in $\mathbb{W}_k$ is uniformly locally finite.
\end{thm}

\begin{proof}
Let $\langle \mathbf{A}, \exists, \forall\rangle$ be a f.s.i.\ monadic Gödel algebra of width $k$. Let $\varphi\colon \mathbf{A} \to \mathbf{A}_1 \times \ldots \times \mathbf{A}_k$ be a subdirect representation of $\mathbf{A}$ as given by Corollary \ref{COR: representacion de algebras de ancho k}. Without loss of generality assume that $\pi_j \circ \varphi$ is the identity on $\exists A$, thus $\exists \mathbf{A}$ is a subalgebra of each $\mathbf{A}_j$.

For each $j$ define two {\em partial} operations $\exists_j, \forall_j$ on $\mathbf{A}_j$ by $$\exists_j x := \min\{c \in \exists A: c \geq x\}, \qquad \forall_j x := \max\{c \in \exists A: c \leq x\},$$ provided these elements exist.

Fix $a_1,\ldots,a_n \in A$ and let $\mathbf{B}$ be the subalgebra of $\langle \mathbf{A}, \exists, \forall\rangle$ generated by $\{a_1,\ldots,a_n\}$. We show that $\mathbf{B}$ is finite.

Put $\varphi(a_i) = (a_{i1},\ldots,a_{ik})$, $1 \leq i \leq n$. For each $j$, let $S_j := \{a_{ij}: 1 \leq i \leq n\} \cup \{0,1\}$, and put $C_j := \forall_j S_j \cup \exists_j S_j$. Let $S_j^* := S_j \cup C_1 \cup \ldots \cup C_k$. Observe that $S_j^*$ is finite for every $j$. We claim that $\varphi(B) \subseteq S_1^* \times \ldots \times S_k^*$ which suffices to prove that $B$ is finite. Indeed, we show that $A' := \varphi^{-1}(S_1^* \times \ldots \times S_k^*)$ is a subuniverse of $\langle \mathbf{A}, \exists, \forall\rangle$. It is clear that $A'$ is closed under Gödel operations. Now fix $x \in A'$ and put $\varphi(x) = (x_1,\ldots,x_k)$. Consider the down-sets $D_j := \{c \in \exists A: c \leq x_j\}$, $1 \leq j \leq k$. Then $\bigcap_{j=1}^k D_j = \{c \in \exists A: c \leq x\}$ and $\max \bigcap_{j=1}^k D_j$ exists (in fact, it is precisely $\forall x$). Since the family of down-sets $\{D_j: 1 \leq j \leq k\}$ is a chain (because $\exists \mathbf{A}$ is totally ordered), there is $j_0$ such that $\bigcap_{j=1}^k D_j = D_{j_0}$. Hence $\forall x = \max \bigcap_{j=1}^k D_j = \max D_{j_0} = \forall_{j_0} x_{j_0}$. Now, since $x_{j_0} \in S_{j_0}^*$, we get that $\forall_{j_0} x_{j_0} \in S_j^*$ for $1 \leq j \leq k$, so $\varphi(\forall x) = (\forall x, \ldots, \forall x) = (\forall_{j_0} x_{j_0}, \ldots, \forall_{j_0} x_{j_0}) \in S_1^* \times \ldots \times S_k^*$, which proves that $\forall x \in A'$. Analogously $\exists x \in A'$. This proves that $A'$ is a finite subuniverse of $\mathbf{A}$ containing $B$.

Finally, note that $|B| \leq |S_1^* \times \ldots \times S_k^*| \leq (n+2+k(2n+4))^k$. Thus the size of $n$-generated subalgebras of $\langle \mathbf{A}, \exists, \forall\rangle$ is uniformly bound.
\end{proof}

Applying \cite[Theorem 3.7]{bez01} we get the following corollary.

\begin{cor} \label{COR: ancho k son loc fin}
The variety $\mathbb{W}_k$ is locally finite for every $k$.
\end{cor}

We close this section with an example that shows that there is much more to say about locally finite algebras in $\mathbb{MG}$.

\begin{example}
Let $\mathbf{A} := \langle [0,1]_G, \exists, \forall\rangle$ where $\exists A = [0,1]$, and let $\mathbf{B} := \langle [0,1]_G^\mathbb{N}, \exists, \forall\rangle$ where $\exists B = \{0,1\}$. Observe that both $\mathbf{A}$ and $\mathbf{B}$ are f.s.i.\ algebras in $\mathbb{MG}$; in addition, note that $\mathbf{A} \in \mathbb{W}_1$ but has infinite height; on the other hand, $\mathbf{B} \in \mathbb{H}_2^\exists$ but has infinite width. Now consider the ordinal sum $\mathbf{C} := \mathbf{A} \oplus \mathbf{B}$ where the top element in $\mathbf{A}$ is identified with the bottom element in $\mathbf{B}$. It is straightforward to check that $\mathbf{C}$ is locally finite, but has infinite width and $\exists C$ is also infinite.
\end{example}


\section{Priestley-type topological representation}\label{sec:espacioPriestleyMG}

In this section we give a topological representation for monadic Gödel algebras. As an application, we characterize congruences by means of saturated closed increasing subsets of the dual space. We also describe the subvarieties introduced in the previous section by means of their dual spaces. We will see in the next section that this duality is a most useful tool to characterize free algebras in the variety $\mathbb{W}_1$.

We start by recalling the definitions needed to state the duality. For a poset $\langle X,\leq\rangle$ and $Y \subseteq X$, let $(Y]=\{x\in X : x \leq
y \ \mbox{for some} \ y \in Y \}$ and $[Y)=\{x\in X  :  x \geq y \
\mbox{for some} \ y \in Y \}$. We write $[x)$, $(x]$ instead of
$[\{x\})$, $(\{x\}]$, respectively. We say that $Y$ is {\it
decreasing} if $Y=(Y]$ and {\it increasing} if $Y=[Y)$. 


A triple $\langle X; \leq, \tau\rangle $ is a {\it totally order-disconnected topological space} if $\langle X, \leq\rangle$ is a poset, $\tau$ is a topology on $X$, and for $x$, $y \in X$, if $x \not\leq y$, then there exists a clopen  increasing set $U \subseteq X$ such that $x \in U$ and $y \notin U$. A {\it compact} totally order-disconnected space is called a {\it Priestley space}.

In \cite{Priestley70} it is proved that the category of bounded distributive lattices and homomorphisms is dually equivalent to the category of Priestley spaces and order-preserving continuous functions. More precisely, Priestley defined contravariant functors $\mathbf{I}$ and $\mathbf{X}$ as follows. If $\mathbf{X}$ is a Priestley space, then $\mathbf{I(X)}$ is the lattice of clopen increasing subsets of $\mathbf{X}$ (we denote by $\D(\mathbf{X})$ the universe of this lattice), and for each morphism $f\colon \mathbf{X}\to \mathbf{X'}$, $\mathbf{I}(f)$ is defined by $\mathbf{I}(f)(U) := f^{-1}(U)$ for each $U \in \D(\mathbf{X'})$. If $\mathbf{L}$ is a bounded distributive lattice, we denote by $\X(\mathbf{L})$ the set of prime filters of $\mathbf{L}$. Then $\mathbf{X(L)}$ is the Priestley space obtained by ordering $\X(\mathbf{L})$ by set inclusion and considering the topology generated by the sets of the form $\sigma(a) := \{P\in \X(\mathbf{L}): a\in P\}$ and $\X(\mathbf{L})\setminus
\sigma(a)$ for each $a\in L$. If $h\colon \mathbf{L}\to \mathbf{L'}$ is a homomorphism, then $\mathbf{X}(h)$ is defined by $\mathbf{X}(h)(P) = h^{-1}(P)$ for each $P \in \X(\mathbf{L'})$. It follows that $\sigma\colon \mathbf{L}\to \mathbf{I(X(L))}$ is a lattice isomorphism, and that the mapping $\varepsilon\colon \mathbf{X}\to \mathbf{X(I(X))}$ defined by the formula $\varepsilon(x) := \{U\in \D(\mathbf{X}): x\in U\}$ is a homeomorphism and an order isomorphism.  We refer the reader to \cite{DP02-libro} for basic properties of Priestley spaces and Priestley duality.

Next we define the kind of Priestley spaces that will prove to be the duals of monadic Gödel algebras.

\begin{dfn} An \emph{MG-space} $\mathbf{X}=\langle X; \leq, \tau, E\rangle$ is a Priestley space $\langle X; \leq, \tau\rangle$ enriched with an equivalence relation $E$ defined on $X$ that satisfies the following conditions:
\begin{enumerate}[(c1)]
\item $(Y]$ is clopen for every clopen $Y\subseteq X$.
\item $[x)$ is a chain for every $x\in X$.
\item The relation $E$ satisfies:
\begin{enumerate}[(c3a)]
\item $\exists U \in \D(\mathbf{X})$ for each $U \in \D(\mathbf{X})$, where $\exists U$ is the union of all the equivalence classes that contain an element of $U$,
\item $\forall U \in \D(\mathbf{X})$ for each $U \in \D(\mathbf{X})$, where $\forall U$ is the union of all the equivalence classes that are contained in $U$,
\item the equivalence classes determined by $E$ are closed in $\mathbf{X}$.
\end{enumerate}
\end{enumerate}
\end{dfn}

We denote by $\bar x$ the equivalence class of an element $x$ in a Priestley space enriched with an equivalence relation $E$.

Let us recall that a Priestley space fulfilling condition (c1) is a Heyting space, that is, the Priestley space associated to a Heyting algebra \cite{Esakia-libro,Priestley84}. If, in addition, the space satisfies condition (c2), then it is possible to prove that the category of these spaces and its morphisms is dually equivalent to the category of Gödel algebras and homomorphisms \cite{MonteiroPortugaliaeMathe}. 

Observe also that if a Priestley space is enriched with an equivalence relation $E$ which satisfies conditions (c3a) and (c3c), then the space is a Q-space \cite{Cignoli91}. Cignoli proved that the category of Q-spaces is dually equivalent to the category of bounded distributive lattices with a quantifier $\exists$ (denoted by $\nabla$ in his paper) that satisfies  the identities $\exists 0 \approx 0$, $x \wedge \exists x \approx x$, $\exists(x \wedge \exists y) \approx \exists x \wedge \exists y$, $\exists(x \vee y) \approx \exists x \vee \exists y$, all of them valid in monadic Gödel algebras.

Let $\mathbf{L}$ be a bounded distributive lattice. A unary operation $\forall\colon \mathbf{L} \to \mathbf{L}$ is called an \emph{interior operator} if it satisfies the following identities $\forall 1 \approx 1$, $x \wedge \forall x \approx \forall x$, $\forall \forall x \approx \forall x$ and $\forall (x \wedge y) = \forall x \wedge \forall y$ (see \cite{Bloktesis}). 

\begin{thm}\label{thmref:sigma_forall} Let $\mathbf{L}$ be a bounded distributive lattice with an interior operator $\forall$ and put $E := \{(P,Q) \in \X(\mathbf{L})^{2} : P \cap \forall L= Q \cap \forall L \}$, an equivalence relation over $\X(\mathbf{L})$. The following are equivalent properties:
 \begin{enumerate}[\rm (i)]
	\item Given $P, Q\in \X(\mathbf{L})$ such that $P\cap \forall L \subseteq Q \cap \forall L$, there exists
	$R\in \X(\mathbf{L})$ such that $(P,R)\in E$ and $R\subseteq Q$.
	\item If $P\in \X(\mathbf{L})$, $a\in P$ and $\overline{P}\subseteq \sigma(a)$, then $\forall a \in P$.
	\item For each $a\in L$, $\forall \sigma (a)= \sigma (\forall a)$, where $\forall \sigma(a)$ is the union of the equivalence classes contained in $\sigma(a)$.
	\item For all $a, b\in L$, $\forall(\forall a \vee b)= \forall a\vee \forall b$.
\end{enumerate}
\end{thm}
\begin{proof}

(i) implies (ii). Let $P\in \X(\mathbf{L})$ such that $a\in P$. We will prove that, if $\forall a\notin P$, then there exists $R\in \X(\mathbf{L})$ such that $(P,R)\in E$ and $a\notin R$. For that, let us consider the filter $F$ of $\mathbf{L}$ generated by $P\cap \forall L$ and the principal ideal $J$ of $\mathbf{L}$ generated by $a$ (note that $P \cap \forall L$ is closed under $\wedge$). Then $F \cap J=\emptyset$. Indeed, if $c \in F \cap J$, there is $k \in P \cap \forall L$ such that $k\leq c\leq a$. Then $k=\forall k \leq \forall a$ and this contradicts that $\forall a\notin P$.	By the Prime Filter Theorem, there exists a prime filter $Q$ such that $F \subseteq Q$ and $Q \cap J = \emptyset$. We have that $Q\in\X(\mathbf{L})$ and $P \cap \forall L \subseteq F \cap \forall L \subseteq Q \cap \forall L$. Now by (i) we obtain $R$ in $\X(\mathbf{L})$ such that $(P,R)\in E$ and $R\subseteq Q$. So, $a\notin R$.
	
(ii) implies (iii). Let $P\in \forall \sigma (a)$. By the definition of $\forall$, $a\in P$ and $\overline{P}\subseteq \sigma (a)$. Considering (ii) we have that $\forall a\in P$, and then $P\in \sigma (\forall a)$. Suppose now that $P\in \sigma (\forall a)$. Then $P\in \sigma(a)$. Let	$Q\in \overline{P}$, that is, $P \cap \forall L= Q \cap \forall L$. Since $\forall a\in P \cap \forall L$, we have that $\forall a \in Q$ and then $a\in Q$. Therefore $\overline{P}\subseteq \sigma (a)$, which implies	that $P\in \forall \sigma (a)$.

(iii) implies (iv).  From (iii) and since $\sigma$ is a lattice isomorphism, we have that $\sigma (\forall(\forall a \vee b))= \forall \sigma (\forall a \vee b) = \forall( \forall \sigma(a) \cup \sigma(b))$. Let us see that $\forall( \forall \sigma(a) \cup \sigma(b))= \forall \sigma(a)\cup \forall \sigma (b)$. Indeed, if $P \in \forall( \forall \sigma(a) \cup \sigma(b))$, then $\overline{P} \subseteq \forall \sigma(a) \cup \sigma(b)$. If $\overline{P} \cap \forall \sigma(a)\neq \emptyset$, then $\overline{P} \subseteq \forall \sigma(a)$. On the other hand, if $\overline{P} \cap \forall \sigma(a)= \emptyset$ then $\overline{P} \subseteq \sigma(b)$. In consequence, $P\in  \forall \sigma(a)\cup \forall \sigma(b)$. Let $P\in  \forall \sigma(a)\cup \forall \sigma(b)$. Then $P\in  \forall \sigma(a)$ or $P\in \forall \sigma(b)$. So, $\overline{P} \subseteq \sigma(a)$ or $ \overline{P} \subseteq \sigma(b)$. Clearly, if $\overline{P} \subseteq \sigma(a)$, then $\overline{P} \subseteq \forall \sigma(a)$. Thus, $\overline{P}	\subseteq \forall \sigma(a) \cup \sigma(b)$ and, from the definition of $\forall$, we have that $ P\in \forall( \forall \sigma(a) \cup \sigma(b))$. Once again, from (iii) and since $\sigma$ is a lattice isomorphism, we obtain $\forall \sigma(a)\cup \forall \sigma (b) = \sigma(\forall a\vee \forall b)$. So, $\sigma (\forall(\forall a \vee b))= \sigma(\forall a\vee \forall b)$.

(iv) implies (i). Let $P, Q\in \X(\mathbf{L})$ such that $P\cap \forall L \subseteq Q\cap \forall L$. Consider the filter $F$ generated by $P\cap \forall L$ and the ideal $J$ generated by $(L\setminus Q)\cup (\forall L\setminus P)$. Let us see that $F\cap J =\emptyset$. Indeed, if $c\in F\cap J$, then there exist $a\in L\setminus Q$ and $\forall b\notin P$ such that $c \leq a \vee \forall b$ (note that $\forall L \setminus P$ is closed under $\vee$ by condition (iv)). Also, there exists $d\in P\cap \forall L$ such that $d\leq c$. By (iv), we have that $d=\forall d\leq \forall c\leq \forall (a\vee \forall b) = \forall a\vee \forall b$. Since $d\in P\cap \forall L$, we have that $\forall a\vee \forall b\in P\cap \forall L$. From $a\notin Q$, we obtain that $\forall a\notin Q$. But $P\cap \forall L \subseteq Q\cap \forall L$, so $\forall a\notin P$ and then $\forall b\in P$, which is a contradiction. By the Prime Filter Theorem  there exists a	prime filter $R$ such that $F \subseteq R$ and $R \cap J = \emptyset$. Since $L\setminus Q\subseteq J \subseteq L \setminus R$, we have that $R\subseteq Q$. Let us prove that $(P,R)\in E$. If $c\in P\cap \forall L$, then $c\in F$ and so $c \in R$. Thus, $c\in R\cap \forall L$. Conversely, let $c\in R\cap \forall L$. If $ c\notin P$, then $c\in \forall L\setminus P$ and $c\in J\subseteq L \setminus R$. So, $c\notin R$. Then, $c\in P \cap \forall L $.
\end{proof}

\begin{remark}
Cignoli proved in \cite[Theorem 2.2]{Cignoli91} a theorem analogous to \thmref{thmref:sigma_forall} for bounded distributive lattices with an additive closure operator. We included the proof of \thmref{thmref:sigma_forall} because neither of the results follows directly from the other one, since these algebras are not symmetric and the properties of $\forall$ are not mere consequences of the corresponding properties of $\exists$.
\end{remark}

Let $\mathbf{A}$ be a monadic Gödel algebra and let us consider the enriched Priestley space \linebreak $\langle \X(\mathbf{A}); \subseteq, \tau, E\rangle$ where $E=\{ (P,Q) \in \X(\mathbf{A})^{2} : P \cap
\exists A= Q \cap \exists A \}$. We already know that this space satisfies (c1), (c2), (c3a) and (c3c) (\cite{Cignoli91}, \cite{MonteiroPortugaliaeMathe}). Moreover, from \thmref{thmref:sigma_forall} and the fact that monadic Gödel algebras satisfy the identity $\forall(\forall x \vee y) \approx \forall x \vee \forall y$, we have that (c3b) is also satisfied. So, the next result follows.

\begin{prop} Let $\mathbf{A}\in \mathbb{MG}$. Then $\langle \X(\mathbf{A}); \subseteq, \tau, E\rangle$ is an MG-space.
\end{prop} 

Let $\mathbf{X}$ be an MG-space. Let us consider the lattice $\mathbf{I(X)}$, where we define $U\implica V= X\setminus(U\setminus V]$, for $U, V\in \D(\mathbf{X})$, and, $\exists$ and $\forall $ as in (c3a) and (c3b). 

\begin{lem} \label{propiedadesDX}
The algebra $\langle \D(\mathbf{X});\cap, \cup, \implica, \emptyset, X,  \exists,\forall\rangle $ satisfies:
\begin{itemize}
\item $\langle \D(\mathbf{X});\cap, \cup, \implica, \emptyset, X \rangle$ is a Gödel algebra; and
\item the following identities:
\begin{multicols}{3}
\begin{enumerate}[$(1)$]
\item\label{MM10} $\forall 1 \approx 1$,
\item\label{MM13} $\exists 0 \approx 0$,

\item\label{MM1} $\forall x \implica x \approx 1$,
\item\label{MM7} $x \implica \exists x \approx 1$,

\item\label{MM6} $\forall \exists x \approx \exists x$,
\item\label{MM11} $\exists \forall x \approx \forall x$,

\item\label{MM37} $\forall(x \wedge y) \approx \forall x \wedge \forall y$,
\item\label{MM20} $\exists(x \vee y) = \exists x \vee \exists y$,

\item\label{MM32} $\exists(x \wedge \exists y) \approx \exists x \wedge \exists y$,
\item\label{MM4} $\forall(\exists x \vee y) = \exists x \vee \forall y$.
\end{enumerate} 
\end{multicols}
\end{itemize}
\end{lem}

\begin{proof}
The fact that $\langle \D(\mathbf{X});\cap, \cup, \implica, \emptyset, X \rangle$ is a Gödel algebra follows immediately from the known duality for Gödel algebras (see \cite{DP02-libro, MonteiroPortugaliaeMathe}).  Clearly from the definitions of $\exists $ and $\forall$ we have (\ref{MM10}), (\ref{MM1}), (\ref{MM6}) and (\ref{MM11}). From \cite{Cignoli91}, we have (\ref{MM13}), (\ref{MM7}), (\ref{MM20}) and (\ref{MM32}).
 It only remains to prove (\ref{MM37}) and (\ref{MM4}). Indeed, clearly $\forall (U\cap V)\subseteq \forall U \cap\forall V$. If $x\in \forall U \cap\forall V $, then $\overline{x}\subseteq U$ and $\overline{x}\subseteq V $. So, $\overline{x}\subseteq U\cap V$ and then $x\in \forall (U\cap V)$. 
To see (\ref{MM4}), let $x\in \exists U \cup \forall V$. If $x\in \exists U$ then $\overline{x}\subseteq \exists U$. So, $\overline{x}\subseteq \exists U\cup V$ which means that $x\in \forall (\exists U\cup V)$. On the other hand, if $x\in \forall V$ then $\overline{x}\subseteq V$. So, $\overline{x}\subseteq \exists U\cup V$ and then $x\in \forall (\exists U\cup V)$. For the other inclusion, let $x\in \forall (\exists U\cup V)$. Then $\overline{x}\subseteq \exists U\cup V$. If $\overline{x}\subseteq \exists U$ then $\overline{x}\subseteq \exists U\cup \forall V$. If $\overline{x}\cap \exists U=\emptyset$ then $\overline{x}\subseteq V$. So, $x\in \forall V$.
\end{proof}

From \lemref{propiedadesDX} we have that $\mathbf{I(X)}$ is a monadic Heyting algebra (as defined by Bezhanishvili in \cite{bezhanishvili98}) that also satisfies the prelinearity identity and (M\ref{M4}). From \cite[Theorem 5.9]{CCDVR17}, we obtain the following.

\begin{prop} If $\langle X; \leq, \tau, E\rangle$ is an MG-space then $ \mathbf{I(X)}$ is a monadic Gödel algebra.
\end{prop}
	

Having established the correspondence between objects, we turn now to morphisms.
 
\begin{dfn} Let $\mathbf{X}$ and $ \mathbf{X'}$ be MG-spaces. An \emph{MG-morphism} $f$ from $\mathbf{X}$ to $\mathbf{X'}$ is a continuous order-preserving map $f\colon \mathbf{X}\to \mathbf{X'}$ such that $f([x))= [f(x))$, $\exists f^{-1}(U')= f^{-1}(\exists U')$ and $\forall f^{-1}(U')= f^{-1}(\forall U')$ for every $U'\in \D(\mathbf{X'})$.
\end{dfn}

From the known dualities for Gödel algebras and Q-distributive lattices we have that if $\mathbf{X},\mathbf{X}'$ are MG-spaces and $f\colon \mathbf{X}\to \mathbf{X'}$ is an MG-morphism, then $\mathbf{I}(f)\colon \mathbf{I}(\mathbf{X'}) \to \mathbf{I}(\mathbf{X})$ is a homomorphism. Conversely, if $\mathbf{A},\mathbf{A}'$ are monadic Gödel algebras and $h\colon  \mathbf{A}\to \mathbf{A'}$ is a homomorphism, consider $\mathbf{X}(h)\colon  \mathbf{X}(\mathbf{A'})\to \mathbf{X}(\mathbf{A})$. Again most of the conditions to check follow from the dualities for Gödel algebras and Q-distributive lattices; the only condition that remains to be proven is that $\forall \mathbf{X}(h)^{-1} (\sigma(a)) = \mathbf{X}(h)^{-1}(\forall \sigma(a))$ for every $a \in A$. Indeed, using item (iii) in \thmref{thmref:sigma_forall}
$$ \forall \mathbf{X}(h)^{-1}(\sigma(a)) = \forall \sigma(h(a))= \sigma(\forall h(a))= \sigma(h(\forall a))=\mathbf{X}(h)^{-1}(\sigma(\forall a)) = \mathbf{X}(h)^{-1}(\forall\sigma(a)).$$ 
Thus, $\mathbf{X}(h)$ is an MG-morphism.

Clearly, $\sigma\colon  \mathbf{A} \to \mathbf{I(X(A))}$ is an isomorphism. The mapping $\varepsilon\colon \mathbf{X}\to \mathbf{X(I(X))}$ is a homeomorphism, an order isomorphism and satisfies the condition: $$(x,y)\in E \Leftrightarrow
(\varepsilon(x),\varepsilon(y))\in \{(P, Q)\in \X(\mathbf{I}(\mathbf{X}))\times \X(\mathbf{I}(\mathbf{X})):
P\cap \exists \D(\mathbf{X}) = Q\cap \exists \D(\mathbf{X})\}$$ \cite[Theorem 2.6]{Cignoli91}. Finally, the naturality of $\sigma$ and $\varepsilon$ follows immediately from the original Priestley duality. So, we have proved the following theorem.

\begin{thm}
The categories of monadic Gödel algebras and MG-spaces are dually equivalent. 
\end{thm}

\subsection{Some applications of the duality}

Recall that if $\mathbf{A} $ is a Heyting algebra and $Y$ is a closed increasing subset of $\X(\mathbf{A})$, then $\theta (Y) := \{(a,b)\in A\times A:\sigma(a)\cap Y = \sigma(b)\cap Y\}$ is a congruence on $\mathbf{A}$ (see \cite{Esakia-libro}). Moreover, the correspondence $Y\mapsto \theta(Y)$ establishes an anti-isomorphism from the lattice of closed increasing sets of $\X(\mathbf{A})$ onto the congruence lattice of $\mathbf{A}$. This properties are clearly inherited by Gödel algebras. Next we derive a similar result for monadic Gödel algebras.

Let $\mathbf{X}$ be an MG-space. A subset $Y\subseteq X$ is called \emph{saturated} if $\exists Y=Y$. Clearly, if $Y$ is saturated, then $\forall Y = Y$.

\begin{thm}
Let $\mathbf{A}\in \MG$. Then, $\theta$ is a congruence of $\mathbf{A}$ if and only if $\theta =  \theta(Y) := \{(a,b)\in A\times A:\sigma(a)\cap Y = \sigma(b)\cap Y\}$ for some saturated closed increasing subset $Y$ of $\X(\mathbf{A})$. Consequently, $Y \mapsto \theta(Y)$ is an anti-isomoprhism from the lattice of saturated closed increasing subsets of $X(\mathbf{A})$ onto the lattice of congruences of $\mathbf{A}$.
\end{thm}

\begin{proof}
Let $\theta$ be a congruence of $\mathbf{A}$. We know that there is a closed increasing subset $Y \subseteq \X(\mathbf{A})$ such that $\theta =\theta(Y)$. It only remains to show that $Y$ is saturated. Let us suppose that there exists $P\in \exists Y \setminus Y$. Then, $\overline{P}\not\subseteq Y$. Since $Y$ is increasing, if $Q\in Y$, then we have that $Q\not\subseteq P$. Thus, for each $Q \in Y$, there is $a_Q\in A$ such that $Q\in \sigma (a_Q)$ and $P\notin \sigma (a_Q)$. Then, $$Y\subseteq \bigcup_{Q\in Y} \sigma (a_Q).$$ Since $Y$ is compact, there exists $a\in A$ such that $Y\subseteq \sigma (a)$ and $P\notin \sigma (a)$. From $\sigma (a)\cap Y = Y = \sigma (1)\cap Y$, we obtain that $(a,1)\in\theta$. Let us see that $(\forall a, 1)\notin \theta$, which is a contradiction. Indeed, $\overline{P}\not\subseteq Y$, $P\notin \sigma (a)$ and $P\in \exists Y$. Then, there exists $R\in \overline{P} \cap Y $, that is, $R\notin \sigma (\forall a)\cap Y$. But $R\in \sigma (1)\cap Y=Y$.
		
Conversely, let $Y$ be a saturated closed increasing subset of $\X(\mathbf{A})$ such that $\theta= \theta(Y) $. In particular, we know that $\theta$ is a congruence of the Gödel reduct of $\mathbf{A}$. Moreover, by \cite[Lemma 3.1]{Cignoli91}, we also know that $\theta $ preserves $\exists$. We need to prove that $\theta $ preserves $\forall$. Let $a, b\in A$ such that $\sigma (a)\cap Y=\sigma (b)\cap Y$. By taking into account \thmref{thmref:sigma_forall}, and since $Y$ is saturated, we have that $\forall(\sigma (a)\cap Y)= \forall \sigma(a)\cap \forall Y= \sigma (\forall a)\cap Y$. Analogously $\forall(\sigma (b)\cap Y)= \sigma (\forall b)\cap Y$, which implies what we wanted.\qedhere	
	\end{proof}

In the following theorem we characterize the MG-spaces corresponding to the algebras in the subvarieties introduced in Section \ref{sec: subvar loc fin}. Given an MG-space $\langle X;\leq,\tau,E\rangle$, observe that, since any class $\overline{x}$ is closed, $\min\overline{x} \neq \emptyset$, where $\min\overline{x}$ is the set of minimal elements of $\overline{x}$. Moreover, since principal decreasing subsets of $X$ are also closed, $\overline{x} \subseteq [\min\overline{x})$ (see \cite{DP02-libro}).

\begin{thm}\label{MGspaceC_D_M} Let $\mathbf{A}\in \mathbb{MG}$, $\mathbf{X}(\mathbf{A})$ be its associated MG-space. Then:
\begin{enumerate}[\rm (1)]
\item $\mathbf{A} \in \mathbb{W}_k$ if and only if the equivalence class $\overline{P}$ has at most $k$ minimal elements for every $P \in \X(\mathbf{A})$. 
\item $\mathbf{A}\in \mathbb{W}_1$ if and only if the equivalence class $\overline{P}$ has exactly one minimal element for every $P \in \X(\mathbf{A})$.
\item $\mathbf{A} \in \mathbb{H}_n$ if and only if the chain $[P)$ has at most $n-1$ elements for every $P \in \X(\mathbf{A})$. \item $\mathbf{A} \in \mathbb{H}_n^\exists$ if and only if $[P)/E$ has at most $n-1$ elements for every $P \in \X(\mathbf{A})$.
\end{enumerate}
\end{thm}

\begin{proof}
\begin{enumerate}[(1)]
\item Let $\mathbf{A}$ be a monadic Gödel algebra that satisfies $(\alpha_k)$. Let us suppose that there exists $P\in \X(\mathbf{A})$ such that $\overline{P}$ has $k+1$ minimal elements $\{N_i:1\leq i\leq k+1\}$. For each $i$ and $j$, $i\neq j$, take $a_{ij}\in A$ such that $a_{ij} \in N_i\setminus N_j$, and  let $a_i = \displaystyle \bigwedge_{j\in \{1, \cdots, k+1\}\setminus\{i\} } a_{ij}$. Clearly $a_i \in N_i$, but $a_i \not\in N_j$ for $j \ne i$. Take $b_i= a_i\implica \bigvee_{j\neq i} a_{j}$, $1\leq i\leq k+1$. Then, for $i \ne t$ we have $$b_i \vee b_t = \bigvee_{j\neq i}(a_i\implica  a_{j})\vee  \bigvee_{j\neq t}(a_t\implica  a_{j}) = 1$$ by prelinearity. So, $\bigwedge_{1\leq i<j\leq k+1}\forall (b_i\vee b_j) =1$. On the other hand, note that $b_i \not\in N_i$ for $1 \leq i \leq k+1$, for otherwise $\bigvee_{j \ne i} a_j \in N_i$ and some $a_j \in N_i$ with $j \ne i$. Then $N_i \not\in \sigma(b_i)$, so  $\overline{P}\not\subseteq \sigma(b_i)$ for any $i$. Hence $P \notin \bigcup \forall \sigma(b_i)$, that is, $\bigvee_{i=1}^{k+1} \forall b_i \ne 1$. This contradicts that $\mathbf{A}$ satisfies $(\alpha_k)$.

Conversely, let us suppose now that each equivalence class $\overline{P}$ has at most $k$ minimal elements. Fix $b_1, \ldots, b_{k+1} \in A$. Let us see that $\bigcap_{1\leq i<j\leq k+1}\forall (\sigma(b_i)\cup \sigma(b_j))\subseteq \bigcup_{1\leq j\leq k+1}\forall \sigma(b_j)$. Let $P\in\bigcap_{1\leq i<j\leq k+1}\forall (\sigma(b_i)\cup \sigma(b_j)) $. Then, $\overline{P}\subseteq \sigma(b_i)\cup \sigma(b_j)$ for each $i\neq j$. Observe that if $N\in \min \overline {P}$ and $N\notin\sigma(b_i)$ for some $i$, then $N\in \sigma(b_j)$ for all $j\neq i$. Since we have at most $k$ minimal elements in $\overline{P}$ then there exists $t\in \{1, \cdots, k+1\}$ such that $\min \overline{P}\subseteq \sigma(b_t)$. Thus, $\overline P\subseteq \sigma(b_t)$ and so $P\in  \bigcup_{1\leq j\leq k+1}\forall \sigma(b_j)$.
 
\item It follows from the previous item setting $k=1$.

\item The corresponding result for Gödel algebras is well known. The present one follows immediately since the identity that defines $\mathbb{H}_n$ does not involve the quantifiers.

\item By the previous item, $\mathbf{A} \in \mathbb{H}_n^\exists$ if and only if $[Q)$ has at most $n-1$ elements for every $Q \in \X(\exists \mathbf{A})$ (here $[Q)$ is considered in $\X(\exists \mathbf{A})$). We now show that the last condition is equivalent to the one stated in the theorem.

Suppose $P_1 \subsetneq \ldots \subsetneq P_n$ is a chain of $n$ prime filters in $\mathbf{A}$ such that $(P_i,P_j) \notin E$ for $i \ne j$. It follows immediately that $P_1 \cap \exists A \subsetneq \ldots \subsetneq P_n \cap \exists A$ is a chain of $n$ prime filters in $\exists \mathbf{A}$.

Conversely, let $Q_1 \subsetneq \ldots \subsetneq Q_n$ be a chain of $n$ prime filters in $\exists\mathbf{A}$. We build a chain $P_1 \subsetneq \ldots \subsetneq P_n$ of $n$ prime filters in $\mathbf{A}$ such that $(P_i,P_j) \not\in E$ for $i \ne j$. First let $F$ be the filter in $\mathbf{A}$ generated by $Q_1$ and let $J$ be the ideal in $\mathbf{A}$ generated by $\exists A \setminus Q_1$. Since $F \cap J = \emptyset$, by the Prime Filter Theorem, there is a prime filter $P_1$ in $\mathbf{A}$ such that $F \subseteq P_1$ and $P_1 \cap J = \emptyset$. It follows that $P_1 \cap \exists A = Q_1$. Now assume $P_1 \subsetneq \ldots \subsetneq P_k$, $k < n$, are already defined in such a way that $P_i \cap \exists A = Q_i$ for $1 \leq i \leq k$; we show how to build $P_{k+1} \in \X(\mathbf{A})$ so that $P_{k+1} \cap \exists A = Q_{k+1}$. Let $F$ be the filter generated by $P_k \cup Q_{k+1}$ and let $J$ be the ideal generated by $\exists A \setminus Q_{k+1}$. We claim that $F \cap J = \emptyset$. Indeed, if $x \in  F \cap J$, then there is $p \in P_k, q \in Q_{k+1}$ and $q' \in \exists A \setminus Q_{k+1}$ such that $p \wedge q \leq x \leq q'$. Hence $\exists(p \wedge q) = \exists p \wedge q \leq q'$. But 
$\exists p \in P_k \cap \exists A = Q_k \subseteq Q_{k+1}$. Then $\exists p \wedge q \in Q_{k+1}$, so $q' \in Q_{k+1}$, a contradiction. By the Prime Filter Theorem, there is a prime filter $P_{k+1}$ in $\mathbf{A}$ containing $F$ such that $P_{k+1} \cap J = \emptyset$. From these conditions it follows that $P_k \subseteq P_{k+1}$ and $P_{k+1} \cap \exists A = Q_{k+1}$. \qedhere
\end{enumerate}
\end{proof} 

\section{The subvariety generated by chains}\label{sec:libresenC}

In this section we study the subvariety $\mathbb{W}_1$ generated by the chains in $\mathbb{MG}$. We give first a totally ordered characteristic chain for $\mathbb{W}_1$, that is, a single totally ordered algebra that generates $\mathbb{W}_1$. Then we prove a Glivenko type theorem for $\mathbb{W}_1$ and finally we characterize the free algebra in $\mathbb{W}_1$ with $n$ generators following a method analogous to those given in \cite{AbadMonteiro} and \cite{rueda01}. 


\subsection{A characteristic chain for $\mathbb{W}_1$}

In this section we give a characteristic chain for the variety $\mathbb{W}_1$, that is, a monadic Gödel chain $\mathbf{A}$ such that $V(\mathbf{A}) = \mathbb{W}_1$.

As shown in Section 2, $\mathbb{W}_1$ is a locally finite variety, and, thus, it is generated by its finite members; moreover, $\mathbb{W}_1$ is generated by all finite monadic Gödel chains. Next we give a full description of finite chains and define an infinite chain $\mathbf{A}$ such that every finite chain belongs to $HS(\mathbf{A})$.

Given a Gödel chain $\mathbf{C}$ we know that a subalgebra $\mathbf{C}'$ of $\mathbf{C}$ is the range of the quantifiers of an expansion $\langle \mathbf{C}, \exists, \forall\rangle$ in $\mathbb{MG}$ if and only if conditions (s1), (s2$_\ell$) and (s3) are met (see Section 1). Condition (s3) is trivially satisfied in Gödel algebras and condition (s2$_\ell$) is trivial in chains, since $1$ is join-irreducible. If, in addition, $\mathbf{C}$ is finite, given any subset $X$ of $C$, we can define quantifiers $\exists$ and $\forall $ over $C$ such that $\exists C=X\cup\{0,1\}$ in the following way $$\exists
 a := \min \{c \in X\cup\{0,1\}: c \geq a\}, \qquad \forall a := \max \{c \in X\cup\{0,1\}: c \leq a\},$$ for each $a\in C$.

Let $m$ be a non-negative integer. Let $C_m$ be the chain with $m+2$ elements $$0 = a_0 < a_1 < \ldots < a_m < a_{m+1} = 1.$$ Let $X := \{0=b_0, b_1, \ldots, b_r, b_{r+1}=1\}\subseteq C_m$, where $r\leq m$ and $b_i< b_j$ if $i< j$. For each $i\in \{0, \ldots, r\}$, let $(b_{i}, b_{i+1}) := \{a\in C_m: b_{i}<a<b_{i+1}\}$, and let $m_i$ be the cardinal of this set. We denote by $\mathbf{C}_{(m, m_0, \ldots, m_{r})}$ the monadic Gödel algebra $\langle \mathbf{C}_m, \exists , \forall \rangle$ such that $\exists C_m= X$. Observe that $\displaystyle r+\sum_{i=0}^{r}m_i=m$. Clearly, if $\mathbf{C}$ is a finite monadic Gödel chain, then $\mathbf{C}$ is isomorphic to some $\mathbf{C}_{(m, m_0, \ldots, m_{r})}$. 

Let $A := (\mathbb{N}_0\times \mathbb{N}_0)\cup \{\top\}$, where $\mathbb{N}_0 := \mathbb{N}\cup \{0\}$. Let us define on $A$ the following total order: 
\begin{itemize}
 \item $(a,b)< \top$, for any $(a,b) \in \mathbb{N}_0\times \mathbb{N}_0$,
 \item $(a, b)\leq (c,d)$ if and only if $a<c$, or, $a=c$ and $b\leq d$, (lexicographical order)
\end{itemize}
and the following quantifiers 
\[\exists \top=\top \text{ and } 
\exists (a,b) =
 \begin{cases}
 (a,0) &\text{if $b = 0$}\\
 (a+1,0)  &\text{if $b \neq 0$}
 \end{cases}, 
\]
and 
\[\forall \top=\top, \text{ and, } 
\forall (a,b) = (a,0). \]

It is clear that $\mathbf{A}= \langle A; \vee, \wedge, \implica,\exists, \forall, (0,0), 
\top\rangle$ is a monadic Gödel chain. Let us show that every finite chain $\mathbf{C}_{(m, m_0, \ldots, m_{r})}$ is a homomorphic image of a subalgebra of $\mathbf{A}$. 

Let $\mathbf{S}$ be the subalgebra of $\mathbf{A}$ whose subuniverse is given by 
\[S=\{(i, j): 0 \leq i \leq r, 0 \leq j \leq m_{i}\}\cup \{(r+1,0)\}\cup \{ \top\}.\] 

Let us rename the elements of $\mathbf{C}_{(m, m_0, \ldots, m_{r})}$ in the following way
\[\{a_{(0,0)}=0,a_{(0,1)}, \ldots , a_{(0,m_0)}, b_1=a_{(1,0)}, \ldots , a_{(1,m_1)}, \ldots ,
a_{(r,0)}=b_{r}, \ldots , a_{(r,m_r)}, \top =b_{r+1}\}.\] 

If we define $h\colon  \mathbf{S}\to \mathbf{C}_{(m, m_0, \ldots, m_{r})}$ as 
\begin{center}
$h((i,j))=a_{(i,j)}$, \quad $h((r+1, 0))= h(\top) = \top,$ 
\end{center}
then it is straightforward to prove that $h$ is an homomorphism and $\mathbf{C}_{(m, m_0, \ldots, m_{r})}$ is a homomorphic image of $\mathbf{S}$. So, all finite chains of $\mathbb{W}_1$ are in the variety generated by $\mathbf{A}$ and consequently, $\mathbf{A}$ is a characteristic algebra for this variety.

\subsection{A Glivenko-type theorem for $\mathbb{W}_1$}

We recall some definitions of special elements. Given a (monadic) Gödel algebra $\mathbf{A}$, an element $a \in A$ is said to be:
\begin{compactitem}
\item {\em dense} if $\neg a = 0$;
\item {\em regular} if $\neg \neg a = a$;
\item {\em boolean} if $a \vee \neg a = 1$.
\end{compactitem}
We denote by $D(\mathbf{A})$, $Reg(\mathbf{A})$ and $B(\mathbf{A})$ the set of dense, regular and boolean elements of $\mathbf{A}$, respectively.

Let $\mathbf{A}$ be a Gödel algebra. In this case $B(\mathbf{A}) = Reg(\mathbf{A})$ is a subuniverse of $\mathbf{A}$ and we denote the corresponding subalgebra by $\mathbf{Reg}(\mathbf{A})$. Moreover, $D(\mathbf{A})$ is an filter and, by Glivenko's theorem, we have that $\mathbf{Reg}(\mathbf{A})\cong \mathbf{A}/D(\mathbf{A})$ and the map $r\colon \mathbf{A}\to \mathbf{Reg}(\mathbf{A})$ defined by $r(a)= \neg \neg a$ is a surjective homomorphism. Note that, since $r(a) = r(\neg \neg a)$ for all $a\in A$, $\neg \neg a \implica a \in D(\mathbf{A})$.

\begin{lem}
For any monadic Gödel algebra $\mathbf{A}$, $B(\mathbf{A}) = Reg(\mathbf{A})$ is a subuniverse of $\mathbf{A}$.
\end{lem}

\begin{proof}
We already know that $B(\mathbf{A})$ is a subuniverse of the Gödel reduct of $\mathbf{A}$. It remains to show that it is also closed under $\exists$ and $\forall$. Let $b \in B(\mathbf{A})$. Then $\exists b \vee \neg \exists b = \exists b \vee \forall \neg b = \forall(\exists b \vee \neg b) = \forall 1 = 1$, so $\exists b \in B(\mathbf{A})$. Analogously, $\forall b \vee \neg \forall b \geq \forall b \vee \exists \neg b = \forall(b \vee \exists \neg b) = \forall 1 = 1$, so $\forall b \in B(\mathbf{A})$ too.
\end{proof}

We introduce now some definitions pertaining specifically to monadic algebras. Let $\mathbf{A}$ be a monadic Gödel algebra. We denote by $D_\forall(\mathbf{A})$ the set of elements $a \in A$ such that $\forall a \in D(\mathbf{A})$, that is, $D_\forall(\mathbf{A}) = \{a \in A: \neg \forall a =0\}$. It is  straightforward to see that the set $D_\forall(\mathbf{A}) $ is a monadic filter of $\mathbf{A}$. Moreover, $\mathbf{A}/D_\forall(\mathbf{A})\in \mathbb{H}^\exists_2$. Indeed, since $\exists a \implica \neg \neg \exists a = 1 \in D_\forall(\mathbf{A})$ and $\neg \forall (\neg \neg \exists a\implica \exists a) = \neg (\neg \neg \exists a \implica \exists a) = 0$, we have that $\exists a$ and $\neg \neg \exists a$ are identified in the quotient $\mathbf{A}/D_\forall(\mathbf{A})$; thus, $\exists \left( \mathbf{A}/D_\forall(\mathbf{A})\right)$ is a Boolean algebra. Bezhanishvili in \cite[Corollary 6]{bezhanishvili2001} proved a similar result for monadic Heyting algebras by means of Esakia spaces. 

We denote by $Reg_\forall (\mathbf{A})$ the set of elements $a \in A$ such that $\forall a \in Reg(\mathbf{A})$, that is, $Reg_\forall (\mathbf{A})= \{a\in A: \neg\neg\forall a = \forall a\}$. The following lemma shows that for algebras in the subvariety $\mathbb{W}_1$ the set $Reg_\forall(\mathbf{A})$ may be endowed with a structure of monadic Gödel algebra in a natural way. Observe that algebras in $\mathbb{W}_1$ satisfy the identity $\neg \neg x \approx \neg \neg \exists x$.


\begin{lem}
If $\mathbf{A} \in \mathbb{W}_1$, then $\mathbf{Reg_\forall(A)} := \langle Reg_\forall (\mathbf{A}); \vee, \wedge, \implica,\exists_r, \forall, 0,1 \rangle\in\mathbb{H}^\exists_2 \cap \mathbb{W}_1$, where $\exists_r a := \neg\neg a = \neg \neg \exists a $.
\end{lem}

\begin{proof}
Let $a, b \in Reg_\forall (\mathbf{A})$. As $\forall (a\vee b)=\forall a \vee\forall b \in Reg(\mathbf{A})$, we have $a \vee b\in Reg_\forall (\mathbf{A})$, and since $\forall (a\wedge b)=\forall a \wedge\forall b \in Reg(\mathbf{A})$, we also have that $a \wedge b \in Reg_\forall(\mathbf{A})$. It is easy to see that in a monadic Gödel chain, if $	\forall b = \neg \neg \forall b$, then $\forall (a \implica b)=\neg \neg \forall(a \implica b)$; so, $a \implica b\in Reg_\forall(\mathbf{A})$. 
	
Thus, $\langle Reg_\forall(\mathbf{A}); \vee, \wedge, \implica, 0,1 \rangle$ is a subalgebra of $\langle A; \vee, \wedge, \implica, 0,1 \rangle$. It remains to show that	$\langle Reg_\forall (\mathbf{A}); \exists_r, \forall \rangle$ satisfies axioms (M1)-(M4) and (M\ref{axiomac}). Clearly $\exists_r a, \forall a \in  Reg_\forall (\mathbf{A})$, for all \linebreak $a \in  Reg_\forall (\mathbf{A})$, and (M1), (M3), (M4) and (M\ref{axiomac}) hold. To see (M2), since $\forall(a \implica \forall b)=\exists a \implica \forall b$, we need to prove $\exists a \implica \forall b = \exists_r a \implica \forall b$. Indeed:
$$\exists a \implica \forall b = \exists a \implica \neg \neg \forall b = \neg \forall b \implica \neg \exists a = \neg \forall b \implica \neg \exists_r a = \exists_r a \implica \neg \neg \forall b = \exists_r a \implica \forall b.$$
Thus, $\mathbf{Reg_\forall(A)} \in\mathbb{W}_1$. Finally, since $\exists_r Reg_\forall (\mathbf{A})= Reg (\mathbf{A})$, then $\mathbf{Reg_\forall(A)} \in\mathbb{H}^\exists_2$.
\end{proof}

We can now prove a Glivenko theorem for algebras in $\mathbb{W}_1$.

\begin{thm}
If $\mathbf{A} \in \mathbb{W}_1$, then $\mathbf{Reg_\forall (A)} \cong \mathbf{A}/D_\forall(\mathbf{A})$.
\end{thm}

\begin{proof}
We define the map $g\colon \mathbf{A} \to \mathbf{Reg_\forall (A)}$ by $g(a)= a \vee \neg\neg \forall a$. Let us see that $g$ is a homomorphism of $\mathbf{A}$ onto $\mathbf{Reg_\forall (A)}$ such that $g^{-1}(\{1\}) = D_\forall(\mathbf{A})$. Indeed, $g(a \vee b)=$ \linebreak $(a \vee b )\vee \neg\neg \forall (a \vee b )= (a \vee b )\vee \neg\neg  (\forall a \vee \forall b )= (a \vee \neg\neg \forall a)\vee (b \vee \neg\neg \forall b)=g(a)\vee g(b)$. It is easy to see that the following holds in monadic Gödel chains:	$(a \wedge b )\vee \neg\neg \forall (a \wedge b ) = (a \vee \neg\neg \forall a)\wedge (b \vee \neg\neg \forall b)$ and $(a \implica b )\vee \neg\neg \forall (a \implica b ) = (a \vee \neg\neg \forall a)\implica (b \vee \neg\neg \forall b)$. Then, clearly, $g(a \wedge b)= g(a)\wedge g(b)$ and $g(a \implica b)= g(a)\implica g(b)$. Also, $g(\forall a)= \forall a \vee \neg\neg \forall a = \forall(a \vee \neg\neg \forall a)=\forall g(a)$, and, $g(\exists a)= \exists a \vee \neg \neg \forall \exists a = \exists a \vee \neg \neg \exists a =\neg \neg \exists a = \neg \neg\exists a \vee \neg \neg \forall a = \neg \neg\exists (a \vee \neg \neg \forall a ) =\exists_r g(a)$. So $g$ is a homomorphism. Clearly $g$ is onto $\mathbf{Reg_\forall (A)}$, since $g(b)= b \vee \neg\neg \forall b = b \vee \forall b= b$ for $b\in Reg_\forall(A)$.

Finally, if $\neg \forall a =0$, then $g(a)=1$. And, if $1= a \vee \neg\neg \forall a$, then $1 = \forall(a \vee \neg \neg \forall a) = \forall a \vee \neg\neg \forall a = \neg\neg \forall a$, so $a \in D_\forall(\mathbf{A})$. Then, $g^{-1}(\{1\}) = D_\forall(\mathbf{A})$.
\end{proof}

\begin{cor}
If $\mathbf{A} \in \mathbb{W}_1$, then $\mathbf{A}/D_\forall(\mathbf{A})$ is semisimple.
\end{cor}

In summary we have shown that for every $\mathbf{A} \in \mathbb{W}_1$ the quotient algebra $\mathbf{A}/D_\forall(\mathbf{A})$ belongs to $\mathbb{H}_2^\exists \cap \mathbb{W}_1$. Since $\mathbb{H}_2^\exists$ is a discriminator variety, we have that $\mathbf{A}/D_\forall(\mathbf{A})$ is a Boolean product of simple monadic Gödel chains.

%
%

\subsection{Free algebras in $\mathbb{W}_1$}


Having a clear description of free algebras in a variety greatly improves the understanding of the structures at hand. We devote the last part of this section to characterize the free algebra $\mathbf{F}(n)$ in $\mathbb{W}_1$ with $n$ generators.

First, we state some results for any finite monadic Gödel algebra. Let $\mathbf{A}\in \mathbb{W}_1$ be a \emph{finite} algebra. We denote by $\Pi(\mathbf{A})$ the family of join-irreducible elements of $\mathbf{A}$. We know that $P\in \X(\mathbf{A})$ if and only if $P=[p)$ is the filter generated by an element $p \in \Pi(\mathbf{A})$. Let us recall that $\langle\X(\mathbf{A}),\subseteq\rangle$ is dually isomorphic to the ordered set $\langle\Pi(\mathbf{A}), \leq\rangle$. The equivalence relation $E$ defined on $\X(\mathbf{A})$ (see \secref{sec:espacioPriestleyMG}) naturally induces an equivalence relation on $\Pi(\mathbf{A})$ which we will also denote by $E$, by an abuse of notation. Note that if $p,q\in \Pi (\mathbf{A})$, $(p,q) \in E$ iff $[p) \cap \exists A = [q) \cap \exists A$ iff $\exists p= \exists q$. We will use the set $\langle\Pi(\mathbf{A}), E\rangle$ instead of $\langle\X(\mathbf{A}), E\rangle$.
Recall that in any Gödel algebra the family of prime filters which contain a prime filter is a chain. Thus, if $p\in \Pi(\mathbf{A})$ then $(p] \cap \Pi(\mathbf{A})$ is a chain. In the sequel we write $(p]_\Pi$ instead of $(p] \cap \Pi(\mathbf{A})$. Note also that $(p]_\Pi = (p] - \{0\}$.

We say that $p\in \Pi(\mathbf{A})\cap \exists A$ has {\em coordinates} $(m,m_0,\ldots,m_r)$ if \[(p]_\Pi=\{p_1, \ldots, p_{m_0}, p_{m_0+1}, p_{m_0+2},\ldots, p_{m_0+m_1+1}, p_{m_0+m_1+2}, \ldots, p_{m+1}=p\}\] has $m+1$ elements and \[(p]_\Pi \cap \exists A = \{p_{m_0+1}, p_{m_0+m_1+2}, \ldots, p_{m+1}=p\}.\] That is, $(p]_\Pi$ is a chain with $m+1$ elements like  \[(p]_\Pi=\{ \underbrace{p_1, \ldots, p_{m_0}}_{\text{not in } \exists A}, \underbrace{p_{m_0+1}}_{\text{in } \exists A}, \underbrace{p_{m_0+2},\ldots, p_{m_0+m_1+1}}_{m_1 \text{ elements not in }\exists A}, \underbrace{p_{m_0+m_1+2}}_{\text{in }\exists A}, \ldots, \underbrace{p_{m+1}=p}_{\text{in } \exists A}\}.\] 

\begin{remark} \label{OBS: p y coordenadas determinan lo de abajo de p}
If $\bl A \in \mathbb{W}_1$ is finite and $p\in \Pi(\mathbf{A})\cap \exists A$ has coordinates $(m,m_0,\ldots,m_r)$, then $\bl A/[p)\cong \bl C_{(m,m_0,\ldots,m_r)}$ (recall the definition of $\bl C_{(m,m_0,\ldots,m_r)}$ from Section 4.1). Indeed, \linebreak $h\colon \mathbf{A} \to \bl C_{(m,m_0,\ldots,m_r)}$ defined by
\[h(x)= \begin{cases}
       1 & \text{ if } x\in [p), \\ 
       a_i & \text{ if } x\in [p_i)\setminus[p_{i+1}), 1\leq i\leq m, \\
       0 & \text{ if } x\notin [p_1),
       \end{cases}\]
is a surjective homomorphism such that $[p)= h^{-1}(\{1\})$. Observe that the coordinates of a prime $p \in \Pi(\mathbf{A}) \cap \exists A$ fully determine the number of prime elements in $\mathbf{A}$ below $p$ as well as which of them belong to $\exists A$.
\end{remark}

\begin{lem} \label{propiedades_de_Pi(A)}
For every finite algebra $\mathbf{A}$ in $\mathbb{W}_1$  we have that:
\begin{enumerate}[\rm (i)]


\item If $p \in \Pi(\mathbf{A})$, then $\exists p \in \Pi(\mathbf{A})$.

\item $\Pi(\exists \mathbf{A}) = \exists \Pi (\mathbf{A}) = \Pi(\mathbf{A})\cap \exists A$,

\item $\max \Pi(\mathbf{A}) =\max \exists \Pi (\mathbf{A})$,  

\end{enumerate} 
\end{lem}

\begin{proof}
\begin{enumerate}[(i)]

\item Let $p \in \Pi(\mathbf{A})$ and let $a, b\in A$ such that $\exists p = a \vee b$. Then $\exists p = \forall(a \vee b) = \forall a \vee \forall b$, so $p = p \wedge \exists p = (p \wedge \forall a) \vee (p \wedge \forall b)$. Since $p$ is an irreducible element, $p = p \wedge \forall a$ or $p = p\wedge \forall b$. Suppose that $p = p \wedge \forall a$. Then $p \leq \forall a$, and so $\exists p \leq \forall a \leq a \leq \exists p$. In consequence, $\exists p = a$. The other case is similar.
	
\item The inclusion $\exists \Pi (\mathbf{A}) \subseteq \Pi(\mathbf{A})\cap \exists A$ follows from the previous item, and $\Pi(\mathbf{A})\cap \exists A \subseteq \Pi(\exists\mathbf{A})$ is trivial. It remains to show that $\Pi(\exists \mathbf{A}) \subseteq \exists \Pi(\mathbf{A})$. Let $a \in \Pi(\exists \mathbf{A})$. Since $\exists a = a$, it suffices to show that $a \in \Pi(\mathbf{A})$. Let $b, c\in A$ such that $a = b \vee c$. Then $a = \forall a = \forall(b \vee c) = \forall b \vee \forall c$. Thus, $a = \forall b$ or $a =\forall c$. If $a = \forall b$, then $a \leq b \leq b \vee c = a$, and consequently $a = b$. Similarly, if $a=\forall c$, then $a = c$. 

\item If $m \in \max \Pi(\mathbf{A})$, then $\exists m \in \Pi(\mathbf{A})$ by (i). Since $m\leq \exists m$, from the maximality of $m$, we have that $m = \exists m$. On the other hand, if $m\in \max \exists \Pi(\mathbf{A})$ and there is $m'\in \Pi(\mathbf{A})$ such that $m\leq m'$, then $m\leq m'\leq \exists m'$ and from the maximality of $m$ we have that $m= \exists m'$ and so $m = m'$. \qedhere
\end{enumerate}
\end{proof}
  

\begin{remark} \label{OBS: de EPi a Pi}
If we know the poset $\exists \Pi(\mathbf{A})$ and the coordinates of its maximal elements, we can fully determine $\langle\Pi(\mathbf{A}), E\rangle$. Indeed, since $\max \Pi(\mathbf{A}) = \max \exists \Pi(\mathbf{A})$, every $p \in \Pi(\mathbf{A})$ lies below a maximal prime element $m \in \exists \Pi(\mathbf{A})$. Moreover, the coordinates of $m$ completely describe the chain of prime elements below it, including which of those elements belong to $\exists A$ (see Remark \ref{OBS: p y coordenadas determinan lo de abajo de p}). Thus the values of $\exists$ on $\Pi(\mathbf{A})$ are immediate and the equivalence relation $E$ is then easily calculated because $(p,q) \in E$ if and only if $\exists p = \exists q$. 

For example, suppose $\exists \Pi(\mathbf{A})$ is the poset shown in Figure \ref{fig:ejemplo} $(a)$ and assume the coordinates of the maximal elements $m_1, m_2, m_3$ are $(4,1,2)$, $(4,1,1,0)$ and $(4,1,1,0)$, respectively. Then the poset $\Pi(\mathbf{A})$ is the one given in Figure 1 $(b)$, where the elements of $\exists \Pi(\mathbf{A})$ are highlighted. The equivalence relation $E$ is now evident and shown in Figure 1 $(c)$. 

\begin{figure}[h]
\begin{center}
\begin{multicols}{3}
\includegraphics[scale=0.5]{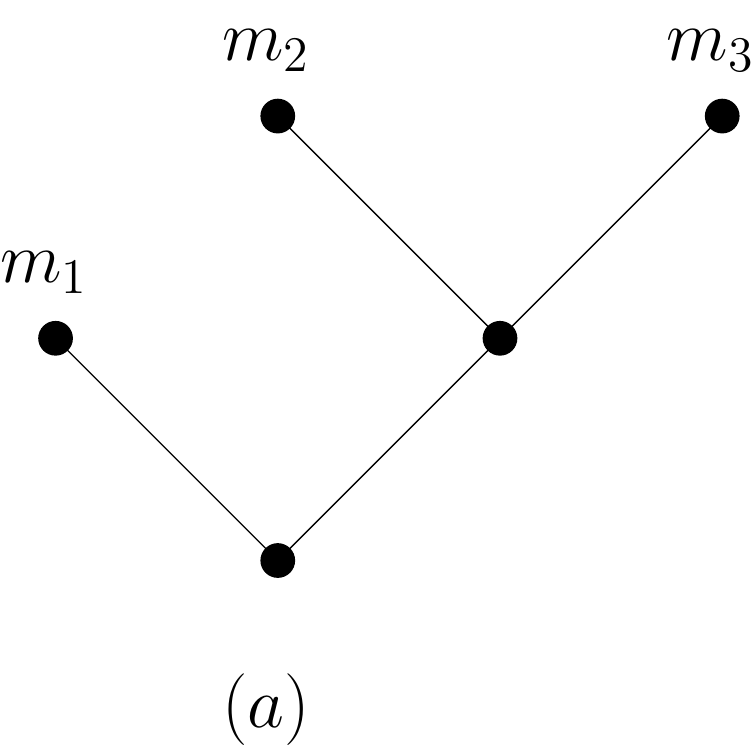}
\includegraphics[scale=0.5]{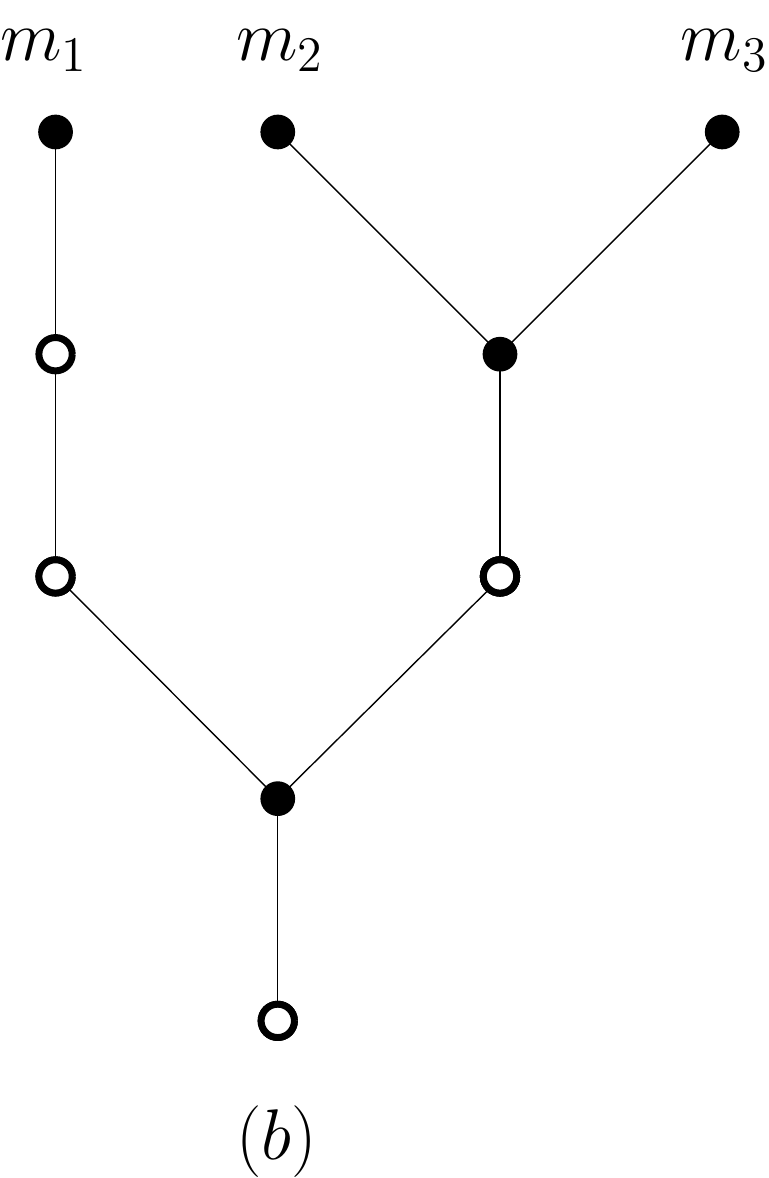}
\includegraphics[scale=0.5]{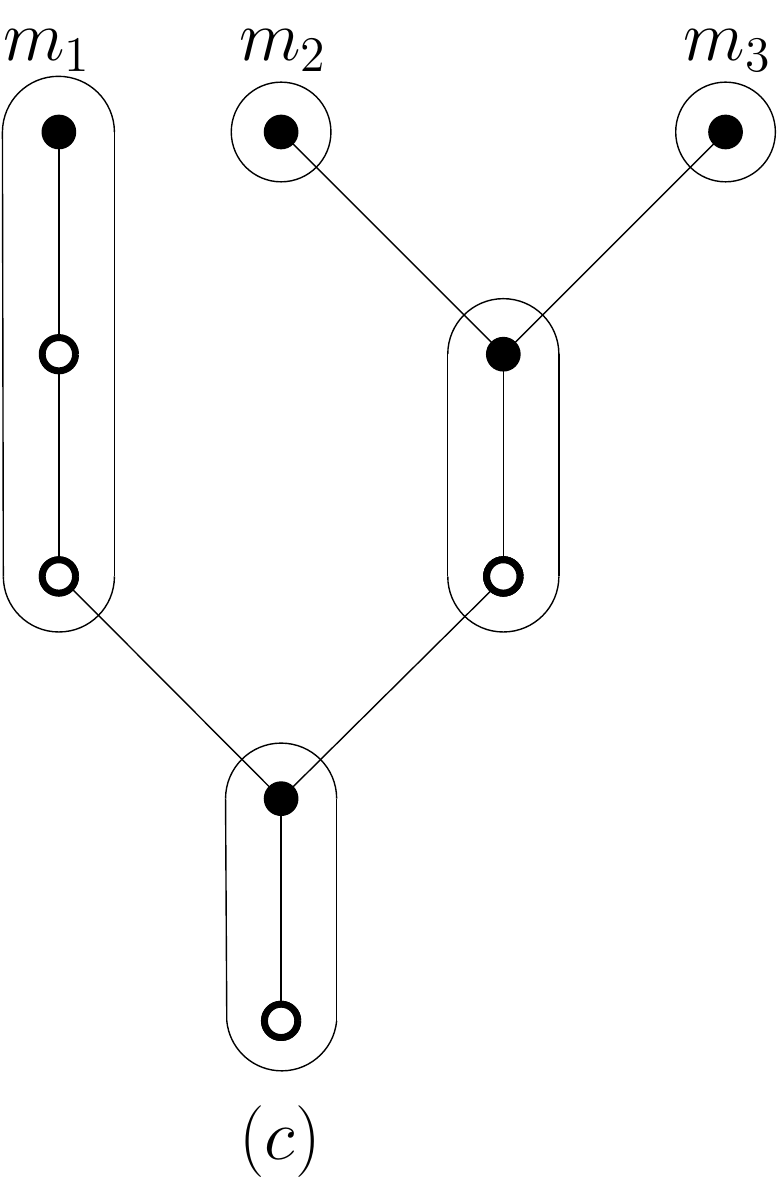}
\end{multicols}
\caption{From $\exists \Pi(\mathbf{A})$ to $\langle\Pi(\mathbf{A}),E\rangle$}
\label{fig:ejemplo}
\end{center}
\end{figure}
\end{remark}

In what follows we characterize the free algebra $\mathbf{F}(n)$ in $\mathbb{W}_1$ with $n$ generators using the procedure described in the last remark. We build the ordered set $\langle\Pi(\mathbf{F}(n)), E\rangle$ of its join-irreducible elements together with the equivalence relation $E$ that determines the quantifiers from the ordered set $\exists \Pi(\mathbf{F}(n))$ and the coordinates of its maximal elements. For the sake of simplicity we write $\Pi(n)$ instead of $\Pi (\mathbf{F}(n)) $ and $\exists \Pi(n)$ instead of $\exists \Pi (\mathbf{F}(n)) $. 

Let $F_{(m,m_0, \ldots, m_{r})}$ be the set of all functions $f$ from the set $G$ of free generators of $ \mathbf F(n)$ into $C_m$ such that the subalgebra generated by $f(G)$ is $\mathbf{C}_{(m, m_0, \ldots, m_{r})}$. If $f\in F_{(m,m_0, \ldots, m_{r})} $, then $f$ can be extended to a unique surjective homomorphism $\bar{f}\colon \mathbf F(n) \to \mathbf{C}_{(m, m_0, \ldots, m_{r})}$ and it is known that $\bar{f}^{-1}(\{1\})=[p_f)$ is a monadic prime filter of $ \mathbf F(n)$  where $p_f\in \exists \Pi(n)$. Moreover, $p_f$ has coordinates $(m,m_0, \dots, m_r)$. On the other hand, if $p\in \exists\Pi(n)$ is such that $\mathbf{F}(n)/[p)\cong \mathbf{C}_{(m, m_0, \ldots, m_{r})}$ and we consider the canonical map $h\colon \mathbf{F}(n)\to \mathbf{F}(n)/[p)$ and the restriction $f=h\restriction_G$, then clearly $f\in F_{(m,m_0, \dots, m_{r})} $, $\bar f= h$ and $p_f=p$. Therefore, there is a bijection between the set $F_{(m,m_0, \ldots, m_{r})}$ and the elements of $\exists \Pi(n)$ with coordinates $(m,m_0, \ldots, m_{r})$. 

Now we want to characterize functions $f\colon G\to C_m$ whose image generates $\mathbf{C}_{(m, m_0, \ldots, m_{r})}$. Let $f\colon G\to C_m$ be such that the subalgebra generated by $f(G)$ is $\mathbf{C}_{(m, m_0, \ldots, m_{r})}$, that is, \linebreak $f(G)\cup \exists f(G)\cup \forall f(G)\cup \{0,1\}= C_m$. Note that, if $|G| = n$, then $m \leq 3n$.

Let \[\exists C_m=\{b_0=0, b_1, b_2, \ldots, b_r, b_{r+1}=1\}\] and let $M$ be the subset of $\exists C_m$ of those elements whose predecessor and successor elements in $C_m$ are both in $\exists C_m$. That is, 
\[M= \{b_j\in \exists C_m: m_{j-1}=m_j=0, 1\leq j\leq r\}.\]

Then, $f(G)$ generates the chain $ \mathbf{C}_{(m, m_0, \ldots, m_{r})}$ if and only if $M\cup \left(C_m\setminus\exists C_m\right) \subseteq f(G)$. So, $|f(G)| \geq |M| + \displaystyle\sum_{i=0}^{r}m_i=|M| + m-r$.

For each $m$, $0\leq m \leq 3n$, let us consider the sets 
\[I_m(n)=\{(m_0, \ldots, m_r): \sum_{i=0}^{r}m_i=m-r \text{ and } m-r+ |\{j: m_{j-1}=m_j=0 \}|\leq n\}.\]

Observe that $|I_m(n)|$ is the number of nonisomorphic chains with $m+2$ elements that can be generated by a set of $n$ generators.

Let $\Lambda (n) = \displaystyle\bigcup_{m=0}^{3n}\bigcup_{(m_0,\dots,m_r)\in I_m(n)} F_{(m,m_0,\dots,m_r)}$. If $f \in \Lambda (n)$, then there is a unique \linebreak $(m_0,\dots,m_r)\in I_m(n)$ such that $f\in F_{(m,m_0,\dots,m_r)}$. Moreover, we have an injection from $\Lambda (n)$ onto $\exists \Pi(n)$ given by $f\mapsto p_f$. Then, each element of $\exists \Pi(n)$ can be represented by an element of $\Lambda (n)$. From the above results we have the following.

\begin{lem} $|\exists \Pi(n)| = \displaystyle\sum_{m=0}^{3n}\sum_{(m_0, \dots, m_r)\in I_m(n)} |F_{(m,m_0, \ldots, m_{r})}|$. 
\end{lem}


\begin{example} \label{EJ: funciones para n = 1}
For $n = 1$ we have $I_0(1) = \{(0)\}$, $I_1(1) = \{(1), (0,0)\}$, $I_2(1) = \{(1,0), (0,1)\}$, and $I_3(1) = \{(0,1,0)\}$.

$F_{(0,0)}$ is the set of functions from $\{g\}$ (set of free generators) into $C_0 = \{a_0,a_1\}$ whose images generate the algebra $\mathbf{C}_{(0,0)}$. There are two choices for the image of $g$ in this case: $a_0$ and $a_1$. We represent those functions by $(0,0;a_0)$ and $(0,0;a_1)$, writing the value of the functions on $g$ after the semicolon. Thus $F_{(0,0)} = \{(0,0;a_0),(0,0;a_1)\}$. In a similar way we can see that $F_{(1,1)} = \{(1,1;a_1)\}$, $F_{(1,0,0)} = \{(1,0,0;a_1)\}$, $F_{(2,1,0)} = \{(2,1,0;a_1)\}$, $F_{(2,0,1)} = \{(2,0,1;a_2)\}$, $F_{(3,0,1,0)} = \{(3,0,1,0;a_2)\}$. 

Consequently $|\exists \Pi(1)| = \displaystyle\sum_{m=0}^{3}\sum_{(m_0, \dots, m_r)\in I_m(1)} |F_{(m,m_0, \ldots, m_{r})}|= 
|F_{(0,0)}|+|F_{(1,1)}|+|F_{(1,0,0)}|+|F_{(2,1,0)}|+|F_{(2,0,1)}|+|F_{(3,0,1,0)}|=7$. 
\end{example}

We say that $q$ covers $p$ if $p<q$ and $p\leq r < q$, implies $p=r$. 

\begin{remark} \label{poset existsPi}
In $\exists \Pi (n)$, $q$ covers $p$ if and only if $p<q$, $p$ has coordinates $(m,m_0, \dots, m_r)$, with $(m_0,\dots,m_r)\in I_m(n)$, and
$q$ has coordinates $(m+m_{r+1}+1,m_0, \dots, m_r, m_{r+1})$, with $(m_0,\dots,m_r,m_{r+1})\in I_{m+m_{r+1}+1}(n)$.
\end{remark}

In particular, from \remarkref{poset existsPi}, we have that $p$ is minimal in the set $\exists \Pi (n)$ if and only if $p$ has coordinates $(m,m)$, with $0\leq m\leq n$.

 \thmref{hcoverf} allows us to construct the ordered set $\exists \Pi(n)$ and determine the coordinates of all its elements. We follow a similar argument given in the proofs of \cite[Theorem 3.14]{AbadMonteiro} and \cite[Theorem 3.10]{rueda01}.
 
\begin{thm}\label{hcoverf} Let $f$, $h\in \Lambda (n)$. Then $p_h$ covers
$p_f$ in $\exists \Pi (n)$ if and only if $f \in F_{(m, m_0, \dots , m_r)}$, where $0 \leq m \leq 3n-1$ and $(m_0, \dots , m_r) \in I_m(n)$, and $h \in F_{(m+m_{r+1}+1,m_0, \dots, m_r, m_{r+1})}$, where $m+m_{r+1}+1\leq 3n$ and $(m_0, \dots , m_r,m_{r+1}) \in I_{m+m_{r+1}+1}(n)$, and, for $g\in G$ the following conditions hold:
\begin{enumerate}[(a)]
\item\label{thm_condition1} $f(g) = a_i$ if and only if $h(g) = a_i$, $0\leq i \leq  m$.
\item\label{thm_condition2} $f(g) =1=a_{m+1}$ if and only if $h(g) = a_i$, $m+1\leq i\leq m+m_{r+1}+1$.
\end{enumerate}
 \end{thm}
\begin{proof}
 If $p_h$ covers $p_f$ in  $\exists \Pi (n)$, then 
$p_1<\ldots <p_m<p_{m+1}=p_f<\ldots<p_{m+m_{r+1}+2}=p_h$ in $\Pi (n)$ and the natural homomorphisms $\bar{f}$, $\bar{h}$ are defined in the following way:

\[\bar{f}(x)= \begin{cases}
       1 & \text{ if } x\in [p_{m+1}), \\ 
       a_i & \text{ if } x\in [p_i)\setminus[p_{i+1}), 1\leq i\leq m, \\
       0 & \text{ if } x\notin [p_1),
       \end{cases}\]
       
\[\bar{h}(x)= \begin{cases}
       1 & \text{ if } x\in [p_h), \\ 
       a_i & \text{ if } x\in [p_i)\setminus[p_{i+1}), 1\leq i\leq m+m_{r+1}+1, \\
       0 & \text{ if } x\notin [p_1).
       \end{cases}\]


In particular, we have that  $f \in F_{(m, m_0, \dots , m_r)}$, $h \in F_{(m+m_{r+1}+1,m_0, \dots, m_r, m_{r+1})}$ and conditions (\itemref{thm_condition1}) and (\itemref{thm_condition2}) hold.

Conversely, let $f \in F_{(m, m_0, \dots , m_r)}$, $h \in F_{(m+m_{r+1}+1 , m_0,\dots , m_{r+1})}$ satisfying (a) and (b). Then, $p_f$ has coordinates $(m, m_0, \dots , m_r) $ and $ p_h$ has coordinates $(m+m_{r+1}+1, m_0, \dots , m_r, m_{r+1})$. From Remark \ref{poset existsPi}, we need to prove that $p_f < p_h$.

Consider in $\Pi(n)$ 
\[p_1 <  \ldots < p_m < p_{m+1}= p_f\]
and
\[q_1 < \ldots < q_m < q_{m+1} < \ldots < q_{m+m_{r+1}+2} =p_h\]
the chains  $(p_f]$ and $(p_h]$
respectively, and the following sets:
\[S_{m+m_{r+1}+2} = [p_h) \cap [p_f),\]
\[S_{m+j} = ([q_{m+j}) \setminus [q_{m+j+1})) \cap [p_{m+1}), 1\leq j \leq m_{r+1}+1,\]
\[S_i = ([q_i) \setminus [q_{i+1})) \cap ([p_i) \setminus [p_{i+1})), \  1\leq i \leq m,\]
\[S_0 = F(n)  \setminus ([q_1)\cup [p_1)).\]
Then
\begin{center}
$a \in S_{m+m_{r+1}+2}$ if and only if $\bar{h}(a) = \bar{f}(a) = 1$,
\end{center}
\begin{center}
$a \in S_{m+j}$ if and only if $\bar{h}(a) = a_{m+j}$ and $\bar{f}(a) = 1$, $1\leq j \leq m_{r+1}+1$,
\end{center}
\begin{center}
$a \in S_i$ if and only if $\bar{h}(a) = \bar{f}(a) =a_i$, $0\leq i \leq m$.
\end{center}

It is a routine matter to show that if $S := \bigcup_{k=0}^{m+m_{r+1}+2} S_k$, then $\mathbf{S}$ is a subalgebra of $\mathbf{F}(n)$ and $G \subseteq S$. Consequently $\mathbf{S}= \mathbf{F}(n)$. Then we can write, $[p_h) = [p_h) \cap F(n)=$ \linebreak $[p_h) \cap (\bigcup_{i=0}^{m+m_{r+1}+2}S_k) = [p_h) \cap [p_f)$. Since $\;p_h \neq p_f\;$, we have  $\;p_f < p_h.$
\end{proof} 

\thmref{hcoverf} induces an order in $\Lambda (n)$ isomorphic to that of $\exists \Pi(n)$.

\begin{example}
We already calculated the elements of $\Lambda(1)$ in Example \ref{EJ: funciones para n = 1}. Using \thmref{hcoverf} we can build the corresponding poset, which is shown in \figref{fig:irreduciblesexiste1}. 

\begin{figure}[h]
\begin{minipage}{0.5\linewidth}
\hfill
\includegraphics[scale=0.6]{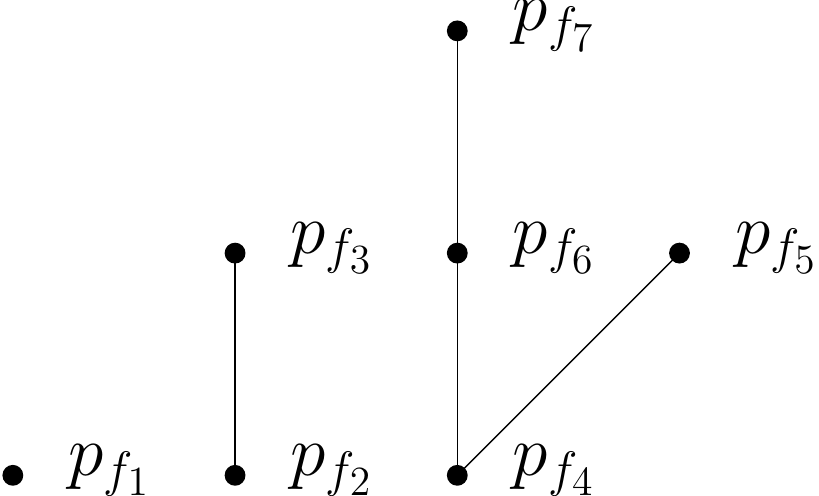}
\end{minipage}
\begin{minipage}{0.5\linewidth}
\hspace{1cm}
\begin{tabular}{ll}
$f_1 := (0,0;a_0)$ \\
$f_2 := (1,1;a_1)$ \\
$f_3 := (2,1,0;a_1)$ \\
$f_4 := (0,0;a_1)$ \\
$f_5 := (1,0,0;a_1)$ \\
$f_6 := (2,0,1;a_2)$ \\
$f_7 := (3,0,1,0;a_2)$ \\
\end{tabular}
\end{minipage}
\caption{$\exists\Pi(1)$}
\label{fig:irreduciblesexiste1}
\end{figure}

Using Remark \ref{OBS: de EPi a Pi} and the coordinates of the maximal elements in $\exists \Pi(1)$ we can build $\langle\Pi(1),E\rangle$. \figref{irredu_free1} shows this poset; we show the decreasing set corresponding to the generator $g$ with a dash line as well as terms for each principal decreasing set.

\begin{figure}[h]
\begin{center}
\includegraphics[scale=0.5]{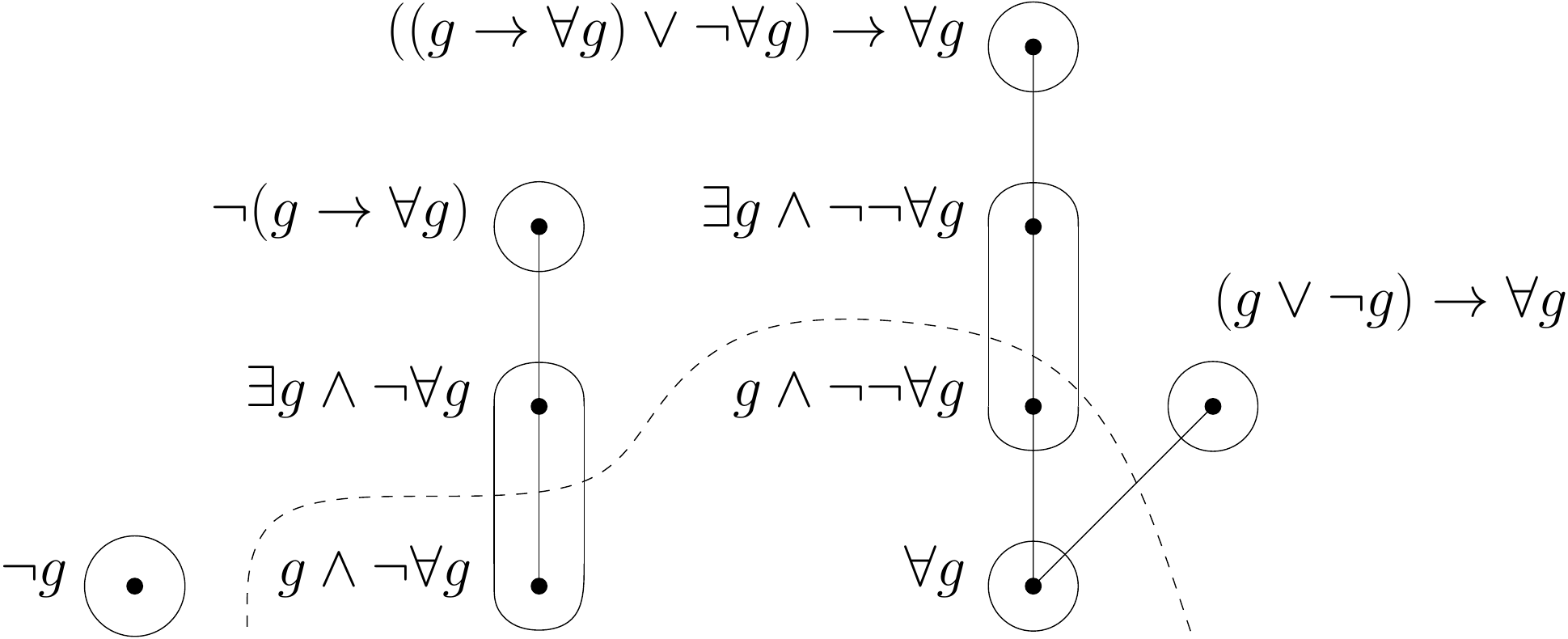}
\caption{$\langle \Pi(1), E\rangle$}
\label{irredu_free1}
\end{center}
\end{figure}
\end{example}

\begin{remark} \label{OBS: observaciones sobre la estructura de la libre}
A few observations on the structure of $\exists \Pi(n)$ can be derived from Theorem \ref{hcoverf}.
\begin{enumerate}
\item $f$ is maximal in $\Lambda (n)$ if and only if $1\notin \exists f(G)$. By \thmref{hcoverf}, it is clear that if $1\notin \exists f(G)$, then there is not $h$ covering $f$. On the other hand, let $f\in F_{(m,m_0,\dots,m_r)}$ such that $1= \exists f(a)$, for some $a\in G$. Let us consider $h\in F_{(m+1,m_0,\dots,m_r,0)}$ defined by $h(g)=f(g)$, if $g\neq a$. If $f(a)=1$, define $h(a)=a_{m+1}$, and if $f(a)\neq 1$, define $h(a)= f(a)$. By \thmref{hcoverf}, $h$ covers $f$.

\item  $f$ is minimal in $\Lambda (n)$ if and only if  $f\in F_{(m,m)}$, with $0\leq m\leq n$.  For example, if $n=1$, then $|\min\Lambda (1)| = |F_{(0,0)}| + |F_{(1,1)}| = 2 + 1 = 3$, and, if $n=2$, then $|\min\Lambda (2)|=|F_{(0,0)}|+|F_{(1,1)}|+|F_{(2,2)}|=4+5+2=11$. More generally, in the case with $n$ free generators, for $0 \leq m \leq n$ we have that $|F_{(m,m)}| = S(n,m) + 2S(n,m+1) + S(n,m+2)$, where $S(n,k)$ is the number of surjective functions from an $n$-element set onto a $k$-element set (recall that $S(n,k) = k! \left\{ n \atop k \right\}$, where $\left\{ n \atop k \right\}$ is a Stirling number of the second king).

\item Let $f\in \min \Lambda (n)$ such that $|f^{-1}(\{1\})|=j$. From \thmref{hcoverf} we know that $h$ covers $f$ if and only if $f\in F_{(m,m)}$ and $h\in F_{(m', m, m'-m-1)}$, with $0\leq m\leq n$ and $m+1\leq m'\leq m+j+1$, and where $f$ and $h$ also satisfy that
\begin{enumerate}
 \item $f(g)= a_i$ if and only if $h(g)= a_i$, for any $i$ such that $0\leq i\leq m$,
 \item $f(g)= 1$ if and only if $h(g)= a_i$, for any $i$ such that $m+1\leq i\leq m'+1$.
\end{enumerate}

If $f_1\in F_{(m'-m-1,m'-m-1)}$ is the function defined by 
\[f_1(g)=
\begin{cases}
0 & \text{ if }  h(g)= a_i, 0\leq i\leq  m, \\
a_{i-(m+1)} & \text{ if }  h(g)= a_i, m+1\leq i\leq m'+1, \\
\end{cases}
\]
then $f_1$ is clearly a minimal element of $\Lambda(n)$. Let us see that $[h)$ and $[f_1)$ are isomorphic. Indeed, if we define $\alpha\colon[h)\to  [f_1)$ by means of $\alpha (u)=v$, where 
\[v(g)=
\begin{cases}
0 & \text{ if } u(g)=a_i, 0\leq i\leq m,\\
a_{i-(m+1)} & \text{ if } u(g)=a_i, m+1\leq i\leq m', \\
\end{cases}
\]
then $\alpha $ is clearly an injection and onto mapping. By \thmref{hcoverf}, it is straightforward to see that $\alpha$ is an isomorphism.

Finally, observe that $\mathbf{F}(n) \cong \mathbf{A}_1\times \mathbf{A}_2$, where $\mathbf{A}_2= \mathbf{0} \oplus \mathbf{A}_1$.

\end{enumerate}
\end{remark} 

\begin{example}
With a little more effort we can calculate the elements of $\Lambda(2)$ and use Theorem \ref{hcoverf} to build the ordered set $\exists \Pi(2)$, which turns out to have 71 elements. From this we can produce the dual space $\langle\Pi(2), E\rangle$ of the free algebra generated by two elements; in this case it has 101 elements. The Hasse diagram is shown in Figure \ref{free_2}.

\begin{figure}[h]
\begin{center}
\includegraphics[scale=0.5]{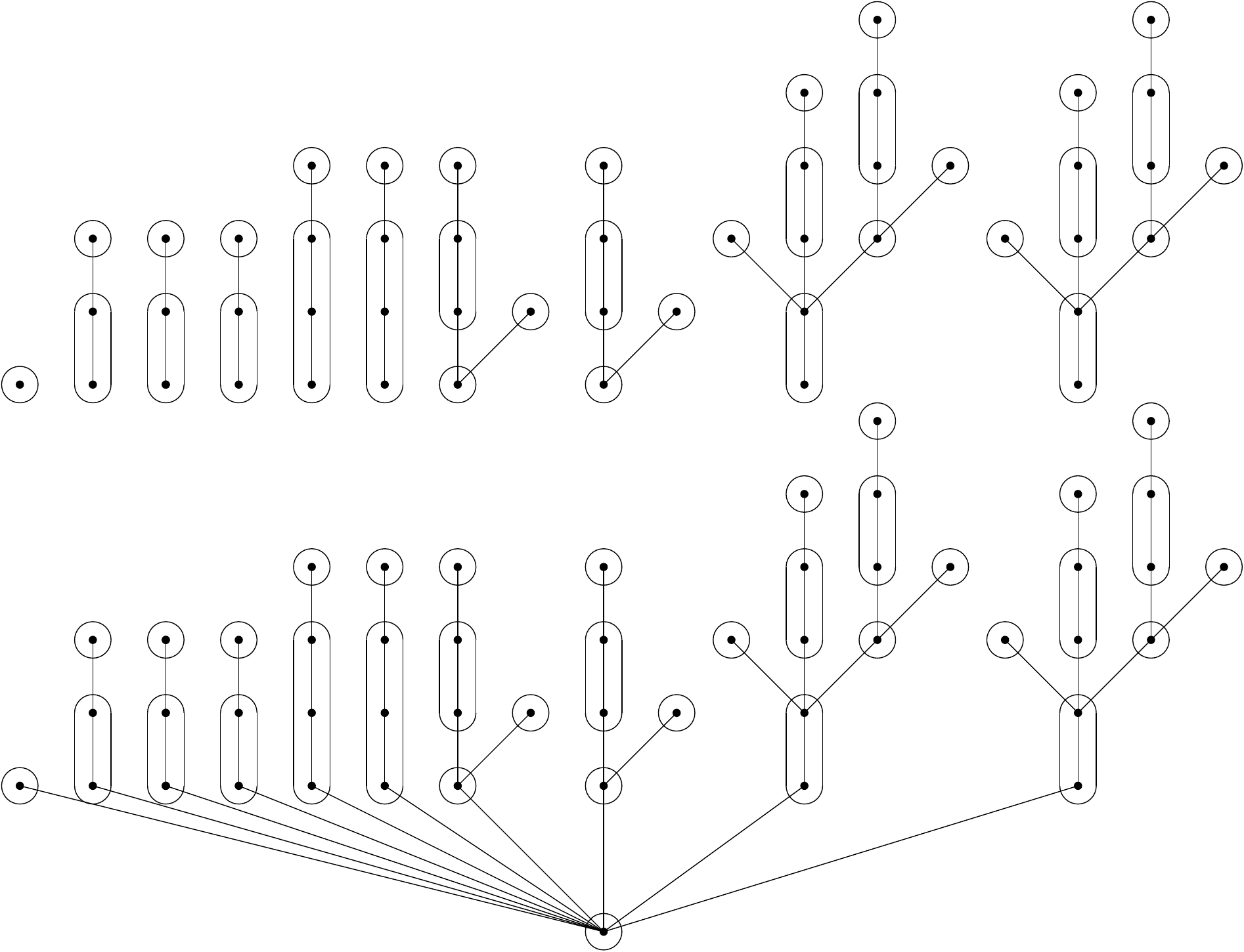}
\caption{$\langle\Pi(2), E\rangle$}
\label{free_2}
\end{center}
\end{figure}
\end{example}

\begin{remark}
Observe that, from the previous work, we can obtain the structure of the \linebreak $n$-generated free algebra $\mathbf{F}_s(n)$ in $\mathbb{H}_2^\exists \cap \mathbb{W}_1$. The subdirectly irreducible algebras in $\mathbb{H}_2^\exists \cap \mathbb{W}_1$ are the simple algebras in $\mathbb{W}_1$, that is, the algebras $\mathbf{C}_{(m,m)}$ for $m \geq 0$. Thus $\exists \Pi(\mathbf{F}_s(n))$ is isomorphic to $\min \Lambda(n)$. Therefore, $$\mathbf{F}(n) \cong \prod_{m=0}^n \mathbf{C}_{(m,m)}^{|F_{(m,m)}|}.$$ See item 2 in Remark \ref{OBS: observaciones sobre la estructura de la libre} for the values of $|F_{(m,m)}|$.
\end{remark}

\

Diego Castaño

\noindent Departamento de Matemática, Universidad Nacional del Sur (UNS), Bahía Blanca, Argentina. \\
Instituto de Matemática (INMABB), Universidad Nacional del Sur (UNS)-CONICET, Bahía Blanca, Argentina.

\noindent diego.castano@uns.edu.ar

\

Cecilia Cimadamore

\noindent Departamento de Matemática, Universidad Nacional del Sur (UNS), Bahía Blanca, Argentina. \\
Instituto de Matemática (INMABB), Universidad Nacional del Sur (UNS)-CONICET, Bahía Blanca, Argentina.

\noindent crcima@criba.edu.ar

\

José Patricio Díaz Varela

\noindent Departamento de Matemática, Universidad Nacional del Sur (UNS), Bahía Blanca, Argentina. \\
Instituto de Matemática (INMABB), Universidad Nacional del Sur (UNS)-CONICET, Bahía Blanca, Argentina.

\noindent usdiavar@criba.edu.ar

\

Laura Rueda

\noindent Departamento de Matemática, Universidad Nacional del Sur (UNS), Bahía Blanca, Argentina. \\
Instituto de Matemática (INMABB), Universidad Nacional del Sur (UNS)-CONICET, Bahía Blanca, Argentina.

\noindent laura.rueda@uns.edu.ar

\end{document}